\documentclass[twoside]{article}

\usepackage{amsmath}
\usepackage{enumerate}
\usepackage{epsfig}
\usepackage{subfigure}
\usepackage{mor}
\usepackage{cases}

\usepackage[mathscr]{eucal}
\usepackage{graphicx}
\usepackage{latexsym}
\usepackage{amssymb}
\usepackage{epsfig}
\DeclareMathAlphabet{\mathpzc}{OT1}{pzc}{m}{it}
\usepackage{enumerate}

\usepackage{amsfonts}

\usepackage{graphicx}

\usepackage{amsfonts}

\usepackage{color}

 \newtheorem{condition}{\bf Condition}[section]

\newcommand{\step}{\lambda}
\newcommand{\stepsum}{\Lambda}

\newcommand{\cla}{\mathcal{A}}

\newcommand{\cls}{\mathcal{S}}
\newcommand{\cln}{\mathcal{N}}
\newcommand{\clp}{\mathcal{P}}
\newcommand{\clm}{\mathcal{M}}
\newcommand{\cll}{\mathcal{L}}

\newcommand{\clf}{\mathcal{F}}
\newcommand{\clg}{\mathcal{G}}

\newcommand{\clc}{\mathcal{C}}
\newcommand{\clb}{\mathcal{B}}

\newcommand{\RR}{\mathbb{R}}
\newcommand{\NN}{\mathbb{N}}

\newcommand{\PP}{\mathbf{P}}
\newcommand{\EE}{\mathbf{E}}

\newcommand{\be}{\begin{equation}}
\newcommand{\ee}{\end{equation}}

\newcommand{\beq}{\begin{eqnarray*}}
\newcommand{\eeq}{\end{eqnarray*}}

\newcommand{\ba}{\begin{aligned}}
\newcommand{\ea}{\end{aligned}}
\newcommand{\bes}{\begin{equation*}}
\newcommand{\ees}{\end{equation*}}
\newcommand{\bsp}{\begin{split}}
\newcommand{\esp}{\end{split}}

\received{}                              %
\revised{}                               %
\pubyear{0x}                             %
\pubmonth{Xxxxxxx}                       %
\volume{xx}                              %
\issue{x}                                %
\pages{xxx--xxx}                         %
\firstpage{0xxx}                         %
\DOI{10.1287/moor.xxxx.xxxx}             %
\startpagenumber{1}                      %

\title{A Numerical Scheme for Invariant Distributions of Constrained Diffusions}
\ShortTitle{Numerical Scheme for Invariant Distributions of Constrained Diffusions}
\ShortAuthors{A. Budhiraja, J. Chen and S. Rubenthaler}
\NumberOfAuthors{3}
\FirstAuthor{Amarjit Budhiraja}
\FirstAuthorAddress{Department of Statistics and Operations
Research, University of North Carolina at Chapel Hill, Hanes
Hall, CB\#3260, Chapel Hill, NC 27599-3260}
\FirstAuthorEmail{budhiraj@email.unc.edu}
\FirstAuthorURL{http://www.unc.edu/~budhiraj/}
\SecondAuthor{Jiang Chen }
\SecondAuthorAddress{Department of Statistics and Operations
Research, University of North Carolina at Chapel Hill, Hanes
Hall, CB\#3260, Chapel Hill, NC 27599-3260}
\SecondAuthorEmail{jiangc@live.unc.edu}
\SecondAuthorURL{}
\ThirdAuthor{Sylvain Rubenthaler }
\ThirdAuthorAddress{Laboratoire de math\'ematiques J.A. Dieudonn\'e,
Universit\'e de Nice-Sophia Antipolis, Parc Valrose, 06108 Nice
cedex 02, France}
\ThirdAuthorEmail{rubentha@unice.fr}
\ThirdAuthorURL{http://math.unice.fr/~rubentha/}

\keywords{Reflected Diffusions, Heavy Traffic Theory, Stochastic Networks, Skorohod Problem, Invariant Measures, Stochastic Algorithms.}

\MSCcodes{Primary:  60J60; Secondary: 60J70, 60K25.    } 

\ORMScodes{Primary: Probability: diffusion, applications; Secondary: Probability: Markov processes, stochastic model applications} 

\begin{document}
\maketitle

\begin{abstract}
	Reflected diffusions in polyhedral domains are commonly used as approximate models for stochastic processing networks in heavy traffic.
	Stationary distributions of such models give useful information on the steady state performance of the corresponding stochastic networks and thus it is important to develop reliable and efficient algorithms for numerical computation of such distributions.  In this work we propose and analyze
	a Monte-Carlo scheme based on an Euler type discretization of the reflected stochastic differential equation using a single sequence of time discretization steps which decrease to zero as time approaches infinity.  Appropriately weighted empirical measures constructed from the  simulated discretized reflected diffusion are proposed as approximations for the  invariant probability measure of the true diffusion model.  Almost sure consistency results are established
	that in particular show that weighted averages of polynomially growing continuous functionals evaluated on the discretized simulated system converge
	a.s. to the corresponding integrals with respect to the invariant measure.  Proofs rely on constructing suitable Lyapunov functions for tightness and uniform integrability and characterizing almost sure limit points through an extension of Echeverria's criteria for reflected diffusions.  Regularity properties of the underlying Skorohod problems play a key role in the proofs.  Rates of convergence for suitable families of test functions are also obtained. A key advantage of Monte-Carlo methods is the ease of implementation, particularly for high dimensional problems.  A numerical example of a eight dimensional Skorohod problem is presented to illustrate the applicability of the approach.
\end{abstract}
\normalsize

\section{Introduction}\label{intr}
Reflected diffusion processes in polyhedral domains have been
proposed as approximate models for critically loaded stochastic
processing networks.  Starting with the influential paper of
Reiman\cite{Reiman84}, there have been many works\cite{Peterson91,
DaiKur95, MandelbaumPats98,Yamada95,Kushner01,Williams98_Diff} that
justify approximations via reflected diffusions rigorously by
establishing a limit theorem under appropriate heavy traffic
assumptions. Many performance measures for stochastic networks are
formulated to capture the long term behavior of the system and a key
object involved in the computation of such measures is the
corresponding steady state distribution. Although classical heavy
traffic limit theorems only justify approximations of the network
behavior through the associated diffusion limit over any fixed
finite time horizon, there are now several results\cite{GamZee,
BudhirajaLee09, BudhirajaLiu102} that prove, for certain generalized
Jackson network models, the convergence of steady state
distributions of stochastic networks to those of the associated
limit diffusions. Such limit theorems then lead to the important
question:  How does one compute the stationary distributions of
reflected diffusions? Indeed, one of the main motivations for
introducing diffusion approximations in the study of stochastic
processing systems is the expectation that diffusion models are
easier to analyze than their stochastic network counterparts.
Classical results of Harrison and Williams \cite{HarWil872} show that
under certain geometric conditions on the underlying problem data,
stationary densities of reflected Brownian motions have explicit
product form expressions.  However, once one moves away from this
special family of models there are no explicit formulas and thus one
needs to use numerical procedures.

The objective of the current work is to propose and study the
performance of one such numerical procedure for computing stationary
distributions of reflected diffusions in polyhedral domains.  For
diffusions in $\RR^m$ there are two basic approaches for computation
of invariant distributions:  PDE methods and Monte-Carlo methods.
PDE approaches are based on the well known basic property that
invariant densities of diffusions can be characterized as solutions
of certain stationary Fokker-Planck equations.  For reflected Brownian motions
in polyhedral domains the papers\cite{DaiKurpre, Kurtz91, DaiHarr92}
develop similar characterization results.  The characterization in
this case is formulated for the invariant density together with
certain boundary densities and is given in terms of the second order
differential operator describing the underlying unconstrained
dynamics and  a collection of first order operators corresponding to
the boundary reflections.  Using this characterization as a starting
point Dai and Harrison\cite{DaiHarr92} develop an approximation
scheme for the stationary density by constructing projections on to
certain finite dimensional Hilbert spaces that are described in
terms of the above collection of differential operators.  Although
PDE methods such as above are quite efficient for settings where the
state dimension $m$ is small, one finds that Monte-Carlo methods,
based on the use of the ergodic theorem, have advantages in higher
dimensions.  With this in mind, we will propose and study here a
Monte-Carlo method for the computation of stationary distributions.
Approximations of invariant distributions of diffusions in $\RR^m$
using simulation of paths have been studied in several works
\cite{Bas,Pell,Talay02,Talay87,LambertonPage02}.
 One of the key difficulties in using
simulation methods is that paths of diffusions cannot be simulated
exactly and so one has to contend with two sources of errors:
Discretization of the SDE and finite time empirical average
approximation for the steady state behavior. In particular, the long
term behavior of the discretized SDE could, in general, be quite
different from that of the original system and thus a performance
analysis of such Monte-Carlo schemes requires a careful
understanding of the stability properties of the underlying systems.

The Monte-Carlo approach studied in the current work is inspired by the papers
\cite{Bas}, \cite{Pell}, \cite{LambertonPage02} which have analyzed the properties
of weighted empirical measures constructed from a Euler scheme,
based on a single sequence of time discretization steps decreasing
to zero, for diffusions in $\RR^m$.  For multi-dimensional
diffusions with reflection one first needs to describe a suitable
analog of an `Euler discretization step'.  In order to do so, we
begin with a precise description of the stochastic dynamical system
of interest.

Let $G \subset \RR^m$ be the convex polyhedral cone in $\RR^m$ with
the vertex at origin given as the intersection of half spaces $G_i$,
$i = 1, \dots, N$. Let $n_i$ be the unit vector associated with
$G_i$ via the relation $$G_i = \{ x \in \RR^m: \langle x, n_i\rangle
\ge 0 \}.$$ Denote the boundary of a set $S \subset \RR^m$ by
$\partial S$. We will denote the set $\{ x \in \partial G: \langle
x, n_i\rangle = 0 \}$ by $F_i$. For $x \in \partial G$, define the
set, $n(x)$, of unit inward normals to $G$ at $x$ by
$$n(x) \doteq \{r: |r| = 1, \langle r, x-y \rangle \le 0, \forall y \in
G\}.$$ With each face $F_i$ we associate a unit vector $d_i$ such
that $\langle d_i, n_i \rangle >0$. This vector defines the
direction of constraint associated with the face $F_i$. For $x \in
\partial G$ define
$$d(x) \doteq \left\{ d\in \RR^m: d = \sum_{i \in \text{In}(x)} \alpha_i d_i ; \alpha_i
\ge 0; |d|=1\right\},$$ where $$\text{In}(x) \doteq \{ i \in \{1,2,
\cdots, N\}: \langle x, n_i\rangle = 0 \}.$$
Roughly speaking, the set $d(x)$ represents the set of permissible directions of constraint available at a point $x \in \partial G$.  In a typical stochastic
network setting this set valued function is determined from the routing structure of the network and governs the precise constraining mechanism that is used.
This mechanism specifies how a RCLL trajectory $\psi$ with values in $\RR^m$ is constrained to form a new trajectory with values in $G$, through the associated
Skorohod problem, which is defined as follows.


 Let $D([0, \infty) :
\RR^m)$ denote the set of functions mapping $[0, \infty)$ to $\RR^m$
that are right continuous and have left limits. We endow $D([0,
\infty) : \RR^m)$ with the usual Skorokhod topology. Let
$$D_G([0, \infty) : \RR^m) \doteq \{\psi \in D([0, \infty) : \RR^m): \psi(0) \in G\}.$$
For $\eta \in D([0, \infty) : \RR^m)$ let $|\eta|(T)$ denote the
total variation of $\eta$ on $[0,T]$ with respect to the Euclidean
norm on $\RR^m$.
\begin{definition} Let $\psi \in D_G([0, \infty) : \RR^m)$ be given. Then the pair $(\phi, \eta) \in
    D([0, \infty) : \RR^m)\times D([0, \infty) : \RR^m)$ solves the Skorokhod problem (SP) for
    $\psi$ with respect to $G$ and $d$ if and only if $\phi(0) =\psi(0)$, and for all
    $t \in [0,\infty)$
    \begin{enumerate}[(i)]
        \item $\phi(t)=\psi(t) +\eta(t)$;
        \item $\phi(t)\in G$;
        \item $|\eta|(t)<\infty$;
        \item $|\eta|(t)=\int_{[0,t]} I_{\{\phi(s)\in \partial G\}}d|\eta|(s)$;
        \item There exists Borel measurable $\gamma:[0,\infty)\rightarrow \RR^m$ such that
        $\gamma(t)\in d(\phi(t))$, $d|\eta|$-almost everywhere and
        $$\eta(t)= \int_{[0,t]}\gamma(s)d|\eta|(s).$$
    \end{enumerate}
\end{definition}
In the above definition $\phi$ represents the constrained version of $\psi$ and $\eta$ describes the correction applied to $\psi$ in order to produce
$\phi$.
On the domain $D \subset D_G([0, \infty) : \RR^m)$ on which there is
a unique solutions
 to the Skorokhod problem we define the Skorokhod map (SM) $\Gamma$ as $\Gamma(\psi)\doteq \phi$, if $(\phi, \psi-\phi)$ is the unique solution of the Skorokhod problem posed by $\psi$. We will make the following assumption on the regularity of the Skorokhod map defined by the data  $\{(d_i,n_i);i =1,2,\cdots,N\}$.
\begin{condition}\label{skolip}
    The Skorokhod map is well defined on all of $D_G([0, \infty) : \RR^m)$, that is, $D=D_G([0, \infty) : \RR^m)$ and the SM is Lipschitz continuous in the following sense. There exists a $K <\infty$ such that for all $\phi_1, \phi_2 \in D_G([0, \infty) : \RR^m)$,
    $$\sup_{0\le t < \infty}|\Gamma(\phi_1)(t)-\Gamma(\phi_2)(t)| < K \sup_{0\le t < \infty}|\phi_1(t)-\phi_2(t)|.$$
\end{condition}

We will also make the following assumption on the problem data.
\begin{condition}\label{completeS}
    For every $x\in \partial G$, there is a $n\in n(x)$ such that
    $\langle d, n \rangle >0$ for all $d\in d(x)$.
\end{condition}
The above condition is equivalent to the assumption that the $N
\times N$ matrix with $(i,j)^{th}$ entry $\langle d_i, n_j \rangle $
is complete-S (see \cite{DupuisIshii91, ReimanWilliams88}). When $G
= \RR_+^m$ and $N=m$, it is known that Condition \ref{skolip}
implies Condition \ref{completeS} (see \cite{TaylorWilliams93}). An
important consequence of Condition \ref{completeS} that will be used
in our work is the following result from \cite{Budhiraja03} (see
also \cite{DaiWilliams95}).

\begin{lemma}\label{lemma:g}
Suppose that Condition \ref{completeS} holds. Then there exists a $g
\in C^2_b (\RR^m)$ such that \be \langle \nabla g(x), d_i \rangle \ge 1
\quad \forall x \in F_i,
 \quad i \in \{1,...,N\}. \ee
\end{lemma}

We remark here that the function constructed in \cite{Budhiraja03} is defined only on $G$, however a minor modification of the construction there gives a $C^2$ extension to all of $\RR^m$.

We refer the reader to \cite{HarrisonReiman81, DupuisIshii91,
DupuisRamanan99} for sufficient conditions under which Condition
\ref{skolip} and Condition \ref{completeS} hold. For example, the
paper \cite{DupuisRamanan99} shows that if $G=\RR_+^m$, $N=m$ and
the square matrix $D=[d_1,...,d_m]$ is of the form
$D=M(I-V)$, where $M$ is a diagonal matrix with positive diagonal entries, $V$ is off diagonal and the
spectral radius of $|V|$ is less than 1,  then both Conditions
\ref{skolip} and \ref{completeS} hold.  Here $|V|$ represents the matrix with entries $(|V_{ij}|)$, where $V_{ij}$ is the $(i,j)$-th entry of $V$.

We now describe the constrained diffusion process that will be
studied in this paper. Let $(\Omega, \clf, \PP)$ be a complete
probability space on which is given a filtration $\{\clf_t\}_{t \ge
0}$ satisfying the usual hypotheses. Let $(W(t), \clf_t)$ be a
$m$-dimensional standard Wiener process on the above probability
space. For $x \in G$, denote by $X^x$ the unique solution to the
following stochastic integral equation, \be \label{consdf} X^x(t)=
\Gamma\left(x + \int_0^{\cdot} \sigma(X^x(s))dW(s)+ \int_0^{\cdot}
b(X^x(s))ds \right)(t),\ee where $\sigma: G \rightarrow \RR^{m\times
m}$ and $b: G \rightarrow \RR^{m}$ are  maps satisfying
the following condition.
\begin{condition} \label{coeff}
    There exists $a_1 \in (0, \infty)$ such that
    \be |\sigma(x)-\sigma(y)|+|b(x)-b(y)|\le a_1|x-y| \quad \forall x,y \in G\ee
    and \be |\sigma(x)| \le a_1,\quad |b(x)| \le a_1,\quad \forall x \in G.\ee
\end{condition}
Unique solvability of \eqref{consdf} can be shown using the above condition and the regularity assumption on the Skorokhod map. In fact, the classical
method of Picard iteration gives the following:
\begin{theorem}
    For each $x \in G$ there exists a unique pair of continuous $\{\clf_t\}$ adapted process $(X^x(t),k(t))_{t \ge 0}$ and a progressively measurable process $(\gamma(t))_{t\ge 0}$ such that the following hold:
    \begin{enumerate}[(i)]
        \item $X^x(t)\in G$, for all $t \ge 0$, a.s.
        \item For all $t \ge 0$, \be \label{ins1928} X^x(t)= x + \int_0^t \sigma(X^x(s))dW(s)+ \int_0^t b(X^x(s))ds +k(t),\ee a.s.
        \item For all $T \in [0,\infty)$,
        $$|k|(T) < \infty \quad \text{a.s.}$$
        \item Almost surely, for every $t \ge 0$,
        $$|k|(t)=\int_0^t I_{\{X^x(s) \in \partial G\}}d|k|(s),$$ $k(t)=\int_0^t \gamma(s)d|k|(s)$, and $\gamma(s)\in d(X^x(s))$ a.e, $[d|k|]$.
    \end{enumerate}
\end{theorem}
In this work we are interested in the invariant distributions of the strong Markov process $\{X^x\}$.  One of the basic results due to
Harrison and Williams\cite{HarWil87} (see also \cite{BudhirajaDupuis99}) on invariant distributions of such Markov processes says that if $b$ and $\sigma$ are constants and $\sigma$ is invertible,
then $X^x$ has a unique invariant probability measure if $b \in \clc^o$ (the interior of $\clc$), where
$$\clc \doteq \left\{-\sum_{i=1}^N \alpha_i d_i: \alpha_i \ge 0; i \in \{1, \cdots, N\}\right\}.$$
This result was extended to a setting with state dependent coefficients in \cite{AtarBudhirajaDupuis01} as follows.
%
We introduce the following two additional assumptions.
For $\delta \in (0,\infty)$, define
$$\clc(\delta)\doteq \{v \in \clc: \text{dist}(v, \partial \clc)\ge \delta\}.$$
\begin{condition}\label{cone}
    There exists a $\delta \in (0,\infty)$ 
    such that for all $x \in G$, 
    $b(x)\in \clc(\delta)$.
\end{condition}
\begin{condition}\label{nondeg}
    There exists $\underline{\sigma} \in (0,\infty)$ such that for all $x \in G$ and $\alpha \in \RR^m$, $$\alpha'(\sigma(x)\sigma'(x))\alpha \ge \underline{\sigma} \alpha'\alpha.$$
\end{condition}
The following is the main result of \cite{AtarBudhirajaDupuis01}.
\begin{theorem} Assume that Conditions \ref{skolip}-\ref{nondeg} hold. Then the strong Markov process $\{X^x(\cdot); x\in G\}$ is positive recurrent and has a unique invariant probability measure.
\end{theorem}
We remark that in \cite{AtarBudhirajaDupuis01} a somewhat weaker assumption than Condition \ref{cone} is used, which says that $b(x)\in \clc(\delta)$
for all $x$ outside a bounded set.  In the current work, for simplicity we will use the stronger form as in Condition \ref{cone}.
Conditions \ref{skolip}-\ref{nondeg} will be assumed to hold for the
rest of this work and will not be explicitly noted in the statements
of various results.

We now summarize some of the  notation that will be used in this work. For a Polish space $S$,
$\mathcal{P}(S)$ will denote the space of probability measures, and
$\mathcal{M}_F(S)$  the space of finite measures on $S$ endowed with
the usual topology of weak convergence. For a closed set $G \subset
\RR^m$, we say $f \in C_b^2(G)$, [respectively $f \in C_c^2(G)$] if
$f$ is defined on some open set $O \supset G$ and $f$ is a twice
continuously differentiable on $O$ with bounded first two
derivatives [respectively compact support]. For $\nu \in
\mathcal{P}(S)$ and a $\nu$-integrable $f: S \rightarrow \RR$, we
write $\int_S f d\nu$ as $\langle f, \nu \rangle$ or $\nu(f)$
interchangeably. We will use the symbol ``$\Rightarrow$'' or ``$\xrightarrow[]{\cll}$'' to denote
convergence in distribution. Let
$\mathbb{R}^m$ denote the set of 
$m$-dimensional real vectors.  Euclidean norm will be denoted by $|\cdot|$ and the corresponding inner product by
$\langle \cdot, \cdot \rangle$. The
symbols, $\xrightarrow[]{\PP}$, $\xrightarrow[]{L^p}$ denote
convergence in  probability and $L^p$ respectively. Denote by
$||\cdot||_\infty$ the supremum norm. A vector $v \in \RR^m$ is said to be nonnegative (and we write $v \ge 0$) if it is componentwise nonnegative.

\subsection{Numerical Scheme and Main Results}
Throughout this work, the unique invariant
measure for the Markov process $\{X^x\}$ will be denoted by $\nu$.
The goal of this work is to develop a convergent numerical procedure
for approximating $\nu$.
We now describe this procedure.

Let $\{\step_k\}_{k\ge 1}$ be a sequence of positive real numbers
such that \be \label{step} \step_k \rightarrow 0, \text{ as }
k\rightarrow \infty \text{ and letting } \stepsum_n := \sum_{k=1}^n
\step_k, \ \stepsum_n \rightarrow \infty \text{ as } n \rightarrow
\infty. 
\ee Note the condition is satisfied if $\step_n =
\frac{1}{n^\theta}$ with $\theta \in (0, 1]$. Define the map $\cls: G
\times \RR^m \rightarrow G$ by the relation
\be \label{ins2000}
\cls(x,v)=\Gamma(x+vi)(1),\ee where $i:[0,\infty) \rightarrow
[0,\infty) $ is the identity map.  The map $\cls$ will be used to construct an Euler discretization of the stochastic dynamical system
described by \eqref{ins1928}.  We now introduce the noise sequence that will be used in the Euler discretization of \eqref{ins1928}.

Let $\{U_{k,j};k \in \NN, j=1,...,m\}$ be an array of mutually
independent $\RR$ valued random variables, given on some probability
space $(\Omega,\clf,\PP)$, such that $\EE U_{k,j}=0$ and $\EE
U_{k,j}^2=1$, for all $k\in\NN, j=1,...,m$. We denote the $\RR^m$
valued random variable $(U_{k,1},...,U_{k,m})'$ by $U_k$. We will
make the following assumption on the array $\{U_{k,j}\}$.

\begin{condition}\label{expint}
For some $\alpha \in (0,\infty)$, \bes \EE e^{\lambda U_{k,j}} \le
e^{\alpha \lambda^2} \text{ for all } k\in \NN, \ j=1,...,m,\
\lambda \in \RR. \ees
\end{condition}

The above condition is clearly satisfied when $U_{k,j} \sim N(0,1)$.
Also, using well known concentration inequalities it can be checked
that the condition also holds if $\text{supp}(U_{k,j})$ is uniformly
bounded (see  Appendix for a proof of the latter statement).
Condition \ref{expint} will be assumed to hold throughout this work.

The Euler scheme is given as follows.  Define iteratively, sequences $\{X_k\}_{k\in \NN_0}$,
$\{Y_k\}_{k\in \NN_0}$ of $G$ and $\RR^m$ valued random variables,
respectively, as follows. Fix $x_0 \in G$.
\begin{equation}
\begin{cases}
X_0=x_0,\\
Y_{k+1} = X_k +b(X_k)\step_{k+1} + \sigma(X_k)\sqrt{\step_{k+1}}U_{k+1},\\
X_{k+1}=\cls(X_k, Y_{k+1}-X_k)\,.\end{cases}\label{scheme2012}\end{equation}
Note that $\{X_k\}$ is a sequence of $G$ valued random variables. The last equation of the above display describes a projection for the Euler step
that is consistent with the Skorohod problem associated with the problem data.

Define a sequence of $\clp(G)$ valued random variables as $$\nu_n =
\frac{1}{\stepsum_n} \sum_{k=1}^n \step_k \delta_{X_{k-1}}, \quad
n\in \NN.$$ The above random measures define our basic sequence of
approximations for the invariant measure $\nu$.  In particular, they yield an approximation for any integral of the form $\int_G f(x) d\nu(x)$
through the corresponding weighted averages:
\be
\label{wtdsum}
\frac{1}{\stepsum_n}\sum_{k=1}^n \step_kf(X_{k-1}).
\ee
The following is the first main result of this work.


\begin{theorem}\label{thm:cgce}
As $n\rightarrow \infty$, $\nu_n$ converges weakly to $\nu$, almost
surely.
\end{theorem}
The above result ensures that \eqref{wtdsum} gives an almost surely consistent approximation for $\nu(f)$ for any bounded and continuous $f$.  In fact we
have a substantially stronger statement as follows:
\begin{theorem}
    \label{thm:unbddmom}
    There exists a $\zeta \in (0, \infty)$ such that for all continuous $f: G \to \RR$ satisfying $\limsup_{x\to \infty} e^{-\zeta |x|} |f(x)| = 0$, we have
    $\nu_n(f) \to \nu(f)$, a.s.
    \end{theorem}
The key ingredient in the proof of the above almost sure limit
theorems is a certain Lyapunov function that was introduced in
\cite{BudhirajaLee07} to study geometric ergodicity properties of
reflected diffusions.  Using this Lyapunov function we establish
a.s. bounds on exponential moments of $\nu_n$ that are uniform in
$n$.  These bounds in particular guarantee tightness of
$\{\nu_n(\omega), n \ge 1\}$, for a.e. $\omega$.  Then the remaining
work, for proving the above theorems, lies in the characterization
of the limit points of $\nu_n(\omega)$.  For this we use an
extension of the well known Echeverria criterion for invariant
distributions of Markov processes that has been developed in
\cite{DaiKurpre, Kurtz91} (see also \cite{Budhiraja03}). Verification of
this criteria (stated as Theorem \ref{thm:inv} in the current work)
for a typical limit point $\nu_0$ of $\{\nu_n\}$ requires showing that,
$\nu_0$ along with a certain collection $\{\mu_0^i, i = 1, \cdots
N\}$ of finite measures supported on various parts of the boundary
of $G$ satisfy a relation of the form in \eqref{ins2045}.  The
measures $\{\mu_0^i\}$ are obtained by taking weak limits of certain
finite measures constructed from the Euler scheme.  Although these
pre-limit measures may place positive mass away from the boundary of the
domain, we argue using the regularity properties of the Skorohod map
(a key ingredient here is Lemma \ref{lemma:g}), that in the limit
these finite measures are supported on the correct parts of the
boundary.


Under additional assumptions, one can obtain rates of convergence as
follows. For $\alpha>0$, set
$$\stepsum_n^{(\alpha)}=\step_1^\alpha +...+ \step_n^\alpha.$$
Denote the normal distribution with mean $a$
and variance $b^2$ by $\cln(a,b^2)$. For $\phi \in C^3(G)$ (space of
three times continuously differentiable functions on $G$) and $v \in
\RR^m$, let $D^3 \phi(x) (v)^{\otimes 3}=\sum_{i,j,k}
D^3_{i,j,k}\phi(x)v_iv_jv_k$.

For $f \in C_c^2(G)$, define $\cla f:G \rightarrow \RR$ and $D_i f:G
\rightarrow \RR$; $i=1,...,N$ as \bes \cla f(x) = b(x) \cdot \nabla
f(x) + \frac{1}{2}\sigma'(x) D^2 f(x)\sigma(x), \quad x \in G,\ees
\bes D_i f(x) =
 d_i \cdot \nabla f(x),\quad x \in G,\ees
 where $\nabla$ is the gradient operator and $D^2$ is the $m\times
 m$ Hessian matrix.

\begin{theorem}\label{thm:rate}Assume
that $U_i$'s are i.i.d with common distribution $\mu$. There exists
a $\zeta \in (0,\infty)$ such that whenever $\phi \in C^2(G)$
satisfies $\lim_{|x| \rightarrow \infty}e^{-\zeta |x|}|\nabla
\phi(x)|^2=0$, we have the following:

(a) Fast-decreasing step. Suppose $\lim_{n \rightarrow
\infty}\frac{\stepsum_n^{(3/2)}}{\sqrt{\stepsum_n}}=0$, $D^2 \phi$ is
bounded and Lipschitz, and
\be \label{DD2}
\begin{cases}
\langle \nabla
\phi(x),d_i\rangle =0, &\forall x \in F_i, \forall i;\\
D^2\phi(x) d_i={\bf 0},&\forall x \in F_i, \forall i.
\end{cases}
\ee Then the following CLT holds: \bes \sqrt{\stepsum_n} \nu_n(\cla
\phi)\xrightarrow[]{\cll} \cln\left(0, \int_G |\sigma^T\nabla
\phi|^2d\nu\right).\ees

(b) Slowly decreasing step. Suppose that $\lim_{n \rightarrow
\infty}(1/\sqrt{\stepsum_n})\stepsum_n^{(3/2)}=\tilde{\step}\in
(0,+\infty]$, $\phi \in C^3(G)$ and $D^3 \phi$ is bounded and
Lipschitz.  Further suppose that \be \label{DD2D3}
\begin{cases}
\langle \nabla
\phi(x),d_i\rangle =0, &\forall x \in F_i, \forall i;\\
D^2\phi(x) d_i={\bf 0},&\forall x \in F_i, \forall i;\\
D^3_{\cdot jk}\phi(x) \cdot d_i=0,&\forall x \in F_i, \forall i,j,k.
\end{cases}
\ee Then we have
\begin{align}\label{slow1}
&\sqrt{\stepsum_n} \nu_n(\cla \phi)\xrightarrow[]{\cll}
\cln\left(\tilde{\step}\tilde{m}, \int_G |\sigma^T\nabla
\phi|^2d\nu\right) &\text{if $\tilde{\step}<
\infty$},\\\label{slow2} &\frac{\stepsum_n}{\stepsum_n^{(3/2)}}
\nu_n(\cla \phi)\xrightarrow[]{\PP} \tilde{m} &\text{if
$\tilde{\step}= +\infty$},
\end{align}
where
$$\tilde{m}=-\frac{1}{6}\int_G\int_{\RR^m}D^3\phi(x)(\sigma(x)u)^{\otimes 3}
\mu(du)\nu(dx).$$
\end{theorem}

Note that when $\step_k=\frac{1}{k^\alpha}$,
${\stepsum_n^{(3/2)}}/{\sqrt{\stepsum_n}}$ converges to 0 [resp.
$\infty$, $\tilde{\step}\in (0,+\infty)$], if $\alpha > 1/2$ [resp.
$\alpha<1/2$, $\alpha=1/2$].  Also note that if $\phi$ is a smooth function supported in the interior of $G$ then it automatically satisfies (\ref{DD2})
and (\ref{DD2D3}).

Proof of Theorem \ref{thm:rate} is quite similar to that of Theorem
9 in \cite{LambertonPage02}, the main difference is in the treatment
of the reflection terms for which once more we appeal to regularity
properties of the Skorohod map and an estimate based on Lemma
\ref{lemma:g} (see proof of \eqref{Lki} which is crucially used in
proofs of Section \ref{ins2107}).

A key step in the implementation of the Euler scheme in \eqref{scheme2012} is the evaluation of the one time step Skorohod map $\cls(x,v)$.  In Section \ref{Sec:eval} we describe one possible approach to this evaluation that uses relationships between Skorohod problems and Linear Complementarity problems(LCPs).
There are many well developed numerical codes for solving LCPs (for example in MATLAB) and we will describe in Section \ref{numerics} some results
from numerical experiments that use a quadratic programming algorithm for LCPs (cf. \cite{CoPaSt09}) in implementing the scheme in \eqref{scheme2012}.  As remarked earlier, one
of the advantages of Monte-Carlo methods is the ease of implementation, particularly for high dimensional problems.  To illustrate this, in Section \ref{numerics} we present numerical results for a eight dimensional Skorohod problem.  

The paper is organized as follows.
In Section \ref{Sec:cgce} we prove  Theorem \ref{thm:cgce} and \ref{thm:unbddmom}. Theorem \ref{thm:cgce} is proved in two steps. Section \ref{sec:tight} shows the tightness of the random measures $\{\nu_n\}$, and Section \ref{sec:limit} characterizes the limit of the measures $\{\nu_n\}$ as the invariant measure of the constrained diffusion in \eqref{consdf}. Section \ref{sec:thmunbdd} gives the proof of Theorem \ref{thm:unbddmom}. Rate of convergence theorem (Theorem \ref{thm:rate}) is proved in Section \ref{ins2107}. Finally we conclude  by describing some numerical results in Section \ref{Sec:sim}.

\section{Proofs of Theorems \ref{thm:cgce} and \ref{thm:unbddmom}}\label{Sec:cgce}

The proof of Theorem \ref{thm:cgce} proceeds by showing that for a.e. $\omega$, the sequence
of random probability measures $\{\nu_n(\omega)\}_{n\ge 1}$ is tight
and then characterizing the limit points of the sequence using a
generalization of  Echeverria's criteria. Tightness is
argued in Section \ref{sec:tight} while the limit points are
characterized in Section \ref{sec:limit}. Finally in Section \ref{sec:thmunbdd}, we give the proof of Theorem \ref{thm:unbddmom}.

\subsection{Tightness} \label{sec:tight}
We begin by presenting a
Lyapunov function introduced in \cite{AtarBudhirajaDupuis01} that
plays a key role in the stability analysis of constrained diffusion
processes of the form studied here (see \cite{AtarBudhirajaDupuis01,
BudhirajaLee07, BudhirajaLee09, BudhirajaLiu102, Budhiraja03,
BudhirajaBorkar04, BudhirajaBiswas}).

Throughout this work we will fix a $\delta >0$ as in Condition \ref{cone}.

For $x \in G$, let $\cla(x)$ be the collection of all absolutely
continuous functions $z:[0,\infty) \rightarrow \RR^m$ defined via
\be z(t) \doteq \Gamma\left(x +\int_0^\cdot v(s)ds\right)(t), \quad
t \in[0,\infty),\ee for some $v:[0,\infty) \rightarrow \clc(\delta)$
which satisfies \be \int_0^t |v(s)|ds < \infty, \quad \text{for all
} t \in [0,\infty). \ee

Define $T: G \rightarrow [0,\infty)$ by the relation \be \label{Tx}
T(x) \doteq \sup_{z \in \cla(x)} \inf\{t \in [0,\infty): z(t)=0\},
\quad x \in G.\ee The function $T$ has the following properties (see
\cite{AtarBudhirajaDupuis01}).
\begin{lemma}\label{propT}
    There exist constants $c,C \in (0,\infty)$ such that the following hold:
    \begin{enumerate}[(i)]
        \item For all $x,y \in G$, $$|T(x)-T(y)|\le C|x-y|.$$
        \item For all $x\in G$, $T(x)\ge c|x|$. Thus, in particular, for all $M \in (0,\infty)$ the set $\{x\in G: T(x) \le M\}$ is compact.
        \item Fix $x\in G$ and let $z \in \cla(x)$. Then for all $t>0$,
        $$T(z(t))\le (T(x)-t)^+.$$
    \end{enumerate}
\end{lemma}

We next present an elementary lemma that will be used in obtaining moment estimates. For $k\in \NN$, let
$\clf_k=\sigma (U_1, ..., U_k)$. Set $\clf_0=\{\emptyset, \Omega\}$.
\begin{lemma}\label{exp}
There exist $c_1, c_2 \in (1, \infty)$ for which the following
holds. Let $\{v_i\}_{i \in \NN}$ be a sequence of $\RR^m$ valued
random variables such that $v_i$ is $\clf_{i-1}$ measurable for all
$i \ge 1$ and \bes \operatorname*{ess\,sup}_{\omega}|v_i(\omega)|
\equiv |v_i|_\infty < \infty.\ees Let $S_n = \sum_{i=1}^n v_i \cdot U_i$,
$n \in \NN$.  Then for every $r \ge 0$ and $n \ge 1$, \bes \EE
\max_{1 \le i \le n} e^{r |S_i|} \le c_1 e^{c_2r^2\sum_{i=1}^n
|v_i|_\infty^2}. \ees
\end{lemma}

\begin{proof}
We will only give the proof for the case $m=1$. The general case is
treated similarly.

From Doob's maximal inequalities for submartingales, we have \bes
\begin{split} \EE \max_{1 \le i \le n}
e^{r |S_i|} &\le 4\EE e^{r |S_n|}\\
&\le 4\left(\EE e^{r S_n}+\EE e^{-r S_n}\right)\end{split}\ees

From Condition \ref{expint},  it follows that for every $r \in \RR$,
\bes
\begin{split} \EE \left(
e^{r S_n}|\clf_{n-1}\right) &\le e^{r S_{n-1}} e^{\alpha r^2 v_n^2}\\
&\le e^{r S_{n-1}} e^{\alpha r^2 |v_n|_\infty^2}\end{split}\ees The
result now follows by a successive conditioning argument.
\end{proof}

Define $\step: [0,\infty) \rightarrow
[0,\infty)$ and $j:[0,\infty) \rightarrow \NN_0$ as
\[\step(s)=\stepsum_k;\;\;\; j(s) = k, \mbox{ if } \stepsum_k \le s <
\stepsum_{k+1}, \; k \in\NN_0;\]
 where we define $\stepsum_0=0$.
Define piecewise linear $\RR^m$ valued stochastic process as
follows, \bes \hat{W}(t)=\sum_{i \le j(t)}\sqrt{\step_i}U_i +
\frac{t-\step(t)}{\sqrt{\step_{j(t)+1}}}U_{j(t)+1}, \quad t \ge
0.\ees Let $\hat{X}(t)$ be the solution of the following integral
equation
$$\hat{X}(t)=\Gamma\left(x_0 + \int_0^\cdot b(\hat{X}(\step(s)))ds +\int_0^\cdot
 \sigma(\hat{X}(\step(s)))d\hat{W}(s) \right)(t), \quad t \ge 0.$$
Clearly, $\hat{X}(\step(t))= X_{j(t)}$ for all $t \ge 0$.

Fix $\rho \in (0,1]$. Define
\be
\label{ins204}\varpi= \frac{1}{2(1+\rho)L},
\quad\Delta =4\step_0 + 16L \ln(c_1),\ee where $L = c_2 a_1^2C^2K^2$
and $\step_0 = \sup_{i \ge 1} \step_i$. Let $V: G \rightarrow
\RR_+$ be defined as \be \label{V} V(x)= e^{\varpi T(x)}, \quad x\in
G.\ee


\begin{lemma}\label{ineq}
There exist $\beta\in (0,1)$ and  $\phi\in [0,\infty)$ such that for
each $\zeta \in [0,\rho]$ and for all $t\ge0$,
    \be \label{ineq1}\EE(V^{1+\zeta}({X}_{j(t+\Delta)})|\clf_{j(t)})\le (1-\beta)V^{1+\zeta}({X}_{j(t)})+\phi \ee
\end{lemma}

\begin{proof}
    Fix $t \ge 0$ and $\zeta \in [0,\rho]$. Define $\xi:[\step(t), \infty) \rightarrow G$ as
    $$\xi(s) = \Gamma\left({X}_{j(t)} + \int_{\step(t)}^{\step(t)+\cdot} b({X}_{j(u)})du\right)(s - \step(t)), \quad s \ge \step(t).$$
    Using the Lipschitz property of the Skorokhod map (Condition \ref{skolip}), we have
    \bes \begin{split}
    \sup_{\step(t)\le s\le \step(t)+\Delta +\step_0} |\hat{X}(s)-\xi(s)| \le &
    K\sup_{\step(t)\le s\le \step(t)+\Delta +\step_0}  \left|\int_{\step(t)}^s \sigma(\hat{X}(\step(u)))d\hat{W}(u)\right|\\
    =: & K\bar{\nu}(t,\Delta).
\end{split}\ees
Note that $$\Delta -\step_0 \le \step(t+\Delta)-\step(t) \le \Delta
+\step_0.$$ Using this observation along with Lemma \ref{propT} (i) and (iii), \bes
\begin{split} T(\hat{X}(\step(t+\Delta))) \le
& T(\xi(\step(t+\Delta))) + CK\bar{\nu}(t,\Delta)\\
\le & (T(\hat{X}(\step(t)))-(\step(t+\Delta)-\step(t)))^+ + CK\bar{\nu}(t,\Delta)\\
\le & (T(\hat{X}(\step(t)))-(\Delta -\step_0))^+ +
CK\bar{\nu}(t,\Delta).
\end{split} \ees
From the above estimate and the definition of $V(x)$, we now have
\begin{align}
\frac{\EE(V(\hat{X}(\step(t+\Delta)))^{1+\zeta}|\clf_{j(t)})}{V(\hat{X}(\step(t)))^{1+\zeta}}
 &\le  \EE\left(\exp(\varpi(1+\zeta)\left((T(\hat{X}(\step(t)))-(\Delta -\step_0))^+ +CK\bar{\nu}(t,\Delta))\right)\left|\clf_{j(t)}\right.\right)\notag\\
& \times \exp(-\varpi(1+\zeta))T(\hat{X}(\step(t))).
 \label{ins207}\end{align}

Letting, for $q \in \NN_0$, $\sigma_q = \sigma(X_q)$, we have, for
any $s \in [\step(t), \step(t)+ \Delta + \step_0]$, \bes
\int_{\step(t)}^s \sigma(\hat{X}(\step(u)))d\hat{W}(u) \le
\begin{cases}
\sum_{q=j(t)}^{j(s)}\sigma_q \sqrt{\step_{q+1}}U_{q+1},& \text{if }
\sigma_{j(s)}U_{j(s)}\ge 0\\ \ \\
\sum_{q=j(t)}^{j(s)-1}\sigma_q \sqrt{\step_{q+1}}U_{q+1},& \text{if
}\sigma_{j(s)}U_{j(s)}< 0\\
\end{cases},
\ees which can be bounded by
$$ \max_{j(t) \le j \le j(s)}\sum_{q=j(t)}^{j}\sigma_q
\sqrt{\step_{q+1}}U_{q+1}.$$ Similarly, \bes -\int_{\step(t)}^s
\sigma(\hat{X}(\step(u)))d\hat{W}(u) \le \max_{j(t) \le j \le
j(s)}-\sum_{q=j(t)}^{j}\sigma_q \sqrt{\step_{q+1}}U_{q+1}. \ees And
therefore
\bes \bar{\nu}(t,\Delta) = \sup_{\step(t)\le s\le \step(t)+\Delta +\step_0}
\left|\int_{\step(t)}^s \sigma(\hat{X}(\step(u)))d\hat{W}(u)\right|
\le \max_{j(t) \le j \le j^*_t}\left|\sum_{q=j(t)}^{j}\sigma_q
\sqrt{\step_{q+1}}U_{q+1}\right|, \ees where
$j^*_t=j(\step(t)+\Delta +\step_0)$.

Using Lemma \ref{exp}, we now have that, with $m_0 = \varpi
(1+\zeta)CK$,
\begin{equation} \label{ab214}
\EE\left[e^{m_0\bar{\nu}(t,\Delta)}\left|\clf_{j(t)}\right.\right] \le c_1 e^{c_2m_0^2 a_1^2 \sum_{q=j(t)}^{j^*_t}\step_{q+1}}
\le c_1 e^{c_2m_0^2 a_1^2 (\Delta + 2 \step_0)}.\end{equation}


In the case $T(\hat{X}(\step(t))) \ge \Delta -\step_0$, we have from (\ref{ins207}) and (\ref{ab214}) that
\bes
\EE(V(\hat{X}(\step(t+\Delta)))^{1+\zeta}|\clf_{j(t)}) \le
V(\hat{X}(\step(t)))^{1+\zeta}e^{-\varpi(1+\zeta)(\Delta
-\step_0)}\times c_1 e^{c_2m_0^2 a_1^2 (\Delta + 2 \step_0)}. \ees
Recalling the choice of $\varpi$ and $\Delta$, we now see that \bes
\EE(V(\hat{X}(\step(t+\Delta)))^{1+\zeta}|\clf_{j(t)}) \le
(1-\beta)V(\hat{X}(\step(t)))^{1+\zeta}, \ees where
$\beta=1-e^{-3\ln c_1}$.

In the case $T(\hat{X}(\step(t))) < \Delta -\step_0$, we have
\[
\EE(V(\hat{X}(\step(t+\Delta)))^{1+\zeta}|\clf_{j(t)}) \le
\EE\left(e^{m_0\bar{\nu}(t,\Delta)}|\clf_{j(t)}\right) \le
c_1 e^{c_2m_0^2 a_1^2 (\Delta + 2 \step_0)} \le
c_1 e^{\frac{1}{4L}(\Delta + 2 \step_0)} \equiv \phi .\]

Combining the two cases, we have \eqref{ineq1}.
\end{proof}

The following lemma follows from Lemma \ref{ineq} through a
recursive argument.

\begin{lemma}\label{Lemma:bnd} There exists $a_2 \in (0,\infty)$ such that
    \be \label{bndrho}\sup_t \EE(V(\hat{X}(\step(t)))^{1+\rho}) \le a_2. \ee
\end{lemma}

\begin{proof}

For any $t \in (\Delta, \infty)$, we can find $t'\in (0,\Delta]$ and
$j \in \NN$ such that $t=t'+j \Delta$. By a recursive argument using
\eqref{ineq1}, we then have
$$\EE(V(\hat{X}(\step(t)))^{1+\rho})\le \EE(V(\hat{X}(\step(t')))^{1+\rho})+ \frac{\phi}{\beta}.$$
Thus \bes \sup_t\EE(V(\hat{X}(\step(t)))^{1+\rho}) \le \sup_{0 \le t
<\Delta}\EE(V(\hat{X}(\step(t)))^{1+\rho})+\frac{\phi}{\beta}. \ees
The supremum on the right side is bounded by $\max_{j \le
j(\Delta+\step_0)}\EE(V({X}_j)^{1+\rho})$, which is finite using
Condition \ref{expint}, boundedness of $b, \sigma$ and the Lipschitz property of $\Gamma$.
\end{proof}

Now we can prove the following lemma.

\begin{lemma}\label{lemma:tight}
For a.e. $\omega$, $\sup_n \langle V, \nu_n(\omega) \rangle <
\infty$. Consequently, the sequence $\{\nu_n (\omega)\}_{n\ge 1}$ is
tight for a.e. $\omega$.
\end{lemma}
\begin{proof}
     Let $n_0$ be such that $\stepsum_{n_0}> \Delta$. Then it suffices to consider the supremum in the above display over
    all $n \ge n_0$. For $i \in \NN_0$, define $s(i) = \inf \{j\in \NN_0: \stepsum_j \ge i \Delta\}$. Then $s(\lfloor \stepsum_n/\Delta\rfloor) \le n$ and therefore, for $n \ge n_0$,
    \bes \begin{split}
    \nu_n(V) =& \frac{1}{\stepsum_n}\sum_{k=1}^n \step_k V(X_{k-1})
    = \frac{1}{\stepsum_n}\int_0^{\stepsum_n} V(\hat{X}(\step(t)))dt \\
    \le& \frac{1}{\stepsum_{s(\lfloor \stepsum_n/\Delta\rfloor)} } \int_0^{(\lfloor \stepsum_n/\Delta\rfloor +1)\Delta}V(\hat{X}(\step(t)))dt .\\
    \end{split} \ees

    Using Lemma \ref{ineq} with $\zeta = 0$, we have
    \bes \begin{split}\frac{1}{\stepsum_{s(\lfloor \stepsum_n/\Delta\rfloor)} } &\int_0^{(\lfloor \stepsum_n/\Delta\rfloor +1)\Delta}V(\hat{X}(\step(t)))dt\\
     &\le \frac{1}{\stepsum_{s(\lfloor \stepsum_n/\Delta\rfloor)} } \int_0^{(\lfloor \stepsum_n/\Delta\rfloor +1)\Delta} [ V(\hat{X}(\step(t))) -\EE(V(\hat{X}(\step(t+\Delta)))|\clf_{j(t)})] dt \\
      &+ \frac{1}{\stepsum_{s(\lfloor \stepsum_n/\Delta\rfloor)} } \int_0^{(\lfloor \stepsum_n/\Delta\rfloor +1)\Delta} [V(\hat{X}(\step(t)))(1-\beta)+\phi] dt.
\end{split} \ees
Thus, rearranging terms,
    \bes \begin{split} \frac{\beta}{\stepsum_{s(\lfloor \stepsum_n/\Delta\rfloor)} } &\int_0^{(\lfloor \stepsum_n/\Delta\rfloor +1)\Delta}V(\hat{X}(\step(t)))dt \\
    &\le \frac{1}{\stepsum_{s(\lfloor \stepsum_n/\Delta\rfloor)} } \int_0^{(\lfloor \stepsum_n/\Delta\rfloor +1)\Delta} [ V(\hat{X}(\step(t))) -\EE(V(\hat{X}(\step(t+\Delta)))|\clf_{j(t)})] dt \\
     &+ \phi \frac{(\lfloor \stepsum_n/\Delta\rfloor
+1)\Delta}{\stepsum_{s(\lfloor \stepsum_n/\Delta\rfloor)}}.
\end{split} \ees

Next note that $\stepsum_{s(\lfloor
\stepsum_n/\Delta\rfloor)} \ge \lfloor \stepsum_n/\Delta\rfloor
\Delta$, and, for $n \ge n_0$,
$$\stepsum_{s(\lfloor \stepsum_n/\Delta\rfloor)} \ge
\stepsum_{s(1)}\ge \step_1.$$ Thus $$\sup_{n\ge n_0}
    \phi\frac{(\lfloor \stepsum_n/\Delta\rfloor
+1)\Delta}{\stepsum_{s(\lfloor \stepsum_n/\Delta\rfloor)}} \le
\phi(1+\frac{\Delta}{\step_1}) < \infty.$$ To prove the lemma, it is
now enough to show \be \label{III}\sup_{n\ge n_0}
\frac{1}{\stepsum_{s(\lfloor \stepsum_n/\Delta\rfloor)} }
\int_0^{(\lfloor \stepsum_n/\Delta\rfloor +1)\Delta} [
V(\hat{X}(\step(t))) -\EE(V(\hat{X}(\step(t+\Delta)))|\clf_{j(t)})]
dt <\infty, \quad \text{a.e. } \omega. \ee

The above expression can be split into two terms:
\bes \begin{split}\frac{1}{\stepsum_{s(\lfloor \stepsum_n/\Delta\rfloor)} } \int_0^{(\lfloor \stepsum_n/\Delta\rfloor +1)\Delta} & [ V(\hat{X}(\step(t))) -\EE(V(\hat{X}(\step(t+\Delta)))|\clf_{j(t)})] dt \\
&= \frac{1}{\stepsum_{s(\lfloor \stepsum_n/\Delta\rfloor)} } \int_0^{(\lfloor \stepsum_n/\Delta\rfloor +1)\Delta} [ V(\hat{X}(\step(t))) -V(\hat{X}(\step(t+\Delta)))] dt \\
 &+\frac{1}{\stepsum_{s(\lfloor \stepsum_n/\Delta\rfloor)} } \int_0^{(\lfloor \stepsum_n/\Delta\rfloor +1)\Delta} [ V(\hat{X}(\step(t+\Delta))) -\EE(V(\hat{X}(\step(t+\Delta)))|\clf_{j(t)})] dt \\
 &\equiv T_1+T_2\end{split}\ees

Consider the first term:
\bes\begin{split} T_1=& \frac{1}{\stepsum_{s(\lfloor \stepsum_n/\Delta\rfloor)} } \int_0^{(\lfloor \stepsum_n/\Delta\rfloor +1)\Delta} V(\hat{X}(\step(t)))  dt - \frac{1}{\stepsum_{s(\lfloor \stepsum_n/\Delta\rfloor)} } \int_0^{(\lfloor \stepsum_n/\Delta\rfloor +1)\Delta} V(\hat{X}(\step(t+\Delta))) dt \\
=& \frac{1}{\stepsum_{s(\lfloor \stepsum_n/\Delta\rfloor)} } \int_0^{(\lfloor \stepsum_n/\Delta\rfloor +1)\Delta} V(\hat{X}(\step(t)))  dt - \frac{1}{\stepsum_{s(\lfloor \stepsum_n/\Delta\rfloor)} } \int_\Delta^{(\lfloor \stepsum_n/\Delta\rfloor +1)\Delta+\Delta} V(\hat{X}(\step(t))) dt \\
=& \frac{1}{\stepsum_{s(\lfloor \stepsum_n/\Delta\rfloor)} } \int_0^{\Delta} V(\hat{X}(\step(t)))  dt - \frac{1}{\stepsum_{s(\lfloor \stepsum_n/\Delta\rfloor)} } \int_{(\lfloor \stepsum_n/\Delta\rfloor +1)\Delta}^{(\lfloor \stepsum_n/\Delta\rfloor +1)\Delta+\Delta} V(\hat{X}(\step(t))) dt \\
\le&\frac{1}{\stepsum_{s(\lfloor \stepsum_n/\Delta\rfloor)} }
\sum_{k:\stepsum_{k-1}\le \Delta}\step_k V(X_{k-1}).
 \end{split}\ees

Let $Z=\sum_{k:\stepsum_{k-1}\le \Delta}\step_k V(X_{k-1})$. Then
from \eqref{bndrho}, we have $\EE Z \le a_2(\Delta+\step_0)$.
Combining this with the fact that for $ n\ge n_0$,
$\stepsum_{s(\lfloor \stepsum_n/\Delta\rfloor)} \ge \step_1$, we
have that \be \label{I} \sup_{n\ge n_0} T_1(\omega) < \infty, \quad
\text{a.e. } \omega. \ee

Next, consider $T_2$:
\bes\begin{split} T_2=& \frac{1}{\stepsum_{s(\lfloor \stepsum_n/\Delta\rfloor)} } \int_0^{(\lfloor \stepsum_n/\Delta\rfloor +1)\Delta} [ V(\hat{X}(\step(t+\Delta))) -\EE(V(\hat{X}(\step(t+\Delta)))|\clf_{j(t)})] dt\\
=& \frac{1}{\stepsum_{s(\lfloor \stepsum_n/\Delta\rfloor)} }
\sum_{i=0}^{\lfloor \stepsum_n/\Delta\rfloor}\int_{i\Delta}^{(i
+1)\Delta} [ V(\hat{X}(\step(t+\Delta)))
-\EE(V(\hat{X}(\step(t+\Delta)))|\clf_{j(t)})] dt.
 \end{split}\ees
From Kronecker's Lemma (see page 63 of \cite{Durrett}), the last sum is bounded in $n$ a.s. (in fact
converges to 0) if the following series is summable a.s. \bes
\sum_{i=1}^{\infty}\frac{1}{\stepsum_{s(i)} } \int_{i\Delta}^{(i
+1)\Delta} [ V(\hat{X}(\step(t+\Delta)))
-\EE(V(\hat{X}(\step(t+\Delta)))|\clf_{j(t)})] dt. \ees

Consider the sum over even and odd terms separately. For even terms,
the sum can be written as \be\label{sum2}
\sum_{k=1}^{\infty}\frac{1}{\stepsum_{s(2k)} } \int_{2k\Delta}^{(2k
+1)\Delta} [ V(\hat{X}(\step(t+\Delta)))
-\EE(V(\hat{X}(\step(t+\Delta)))|\clf_{j(t)})] dt. \ee Let
\[\xi_{k+1}=\frac{1}{\stepsum_{s(2k)} } \int_{2k\Delta}^{(2k
+1)\Delta} [ V(\hat{X}(\step(t+\Delta)))
-\EE(V(\hat{X}(\step(t+\Delta)))|\clf_{j(t)})] dt\]
 and
$\clg_k=\clf_{j(2k\Delta)}$, then we have $\EE(\xi_{i+1}|\clg_i)
=0$. Also note that $\xi_{i+1}$ is $\clg_{i+1}$ measurable. Thus
$S_n = \sum_{i=1}^n \xi_i$ is a martingale with respect to the
filtration $\{\clg_n\}$. Consequently, by Chow's Theorem (see Theorem 2.17 of
\cite{HallHeyde80}), the series in \eqref{sum2} is
a.s. summable if $\sum_{k=1}^\infty \EE (|\xi_k|^{1+\rho}) <
\infty$. Now note that \bes \begin{split} \EE
|\xi_k|^{1+\rho}&=\EE\left(\left|\frac{1}{\stepsum_{s(2k)} }
\int_{2k\Delta}^{(2k +1)\Delta} [ V(\hat{X}(\step(t+\Delta)))
-\EE(V(\hat{X}(\step(t+\Delta)))|\clf_{j(t)})]
dt\right|^{1+\rho}\right)\\
&\le \frac{2^{1+\rho}\Delta^{1+\rho}}{\stepsum_{s(2k)}^{1+\rho} }
\sup_t \EE(V(\hat{X}(t))^{1+\rho}) \le
\frac{2^{1+\rho}\Delta^{1+\rho} a_2}{\stepsum_{s(2k)}^{1+\rho} },
 \end{split}\ees
where the last inequality follows form Lemma \ref{Lemma:bnd}. Since
$\stepsum_{s(k)} \ge k \Delta$, we have that \bes
\sum_{k=1}^{\infty} \frac{1}{\stepsum_{s(k)}^{1+\rho} }\le
\frac{1}{\Delta^{1+\rho}}\sum_{k=1}^{\infty}\frac{1}{k^{1+\rho}} <
\infty .\ees This proves that the series in \eqref{sum2} is
summable.
The odd terms are treated in a similar manner.
Thus we have proved \be \label{II}\sup_{n\ge n_0} T_2(\omega) <
\infty, \quad \text{a.e. } \omega. \ee Now \eqref{III} is an
immediate consequence of \eqref{I} and \eqref{II}, which proves the
lemma.
\end{proof}

\subsection{Identification of the limit}\label{sec:limit}
In this section we will complete the proof of Theorem \ref{thm:cgce}
by arguing that for a.e. $\omega$, every weak limit point of
$\nu_n(\omega)$ equals $\nu$. For this we will use the following
extension of the Echeverria Criteria (see \cite{Kurtz91,
Weiss81}, see also Theorem 5.7 of \cite{Budhiraja03}).

\begin{theorem}\label{thm:inv}
Let $\nu_0 \in \clp(G)$ and $\mu_0^i \in \clm_F(F_i)$, $i=1,...,N$
be such that for all $f \in C_c^2(G)$, \be \label{ins2045}
    \nu_0(\cla f) +\sum_{i=1}^N \mu_0^i(D_i f)=0.
    \ee Then $\nu_0 =\nu$.
\end{theorem}

In order to apply the above theorem to show convergence of $\nu_n$
to $\nu$, we will consider a sequence of finite measure
$\{\mu_n^i\}_{n\in\NN}$; $i=1,...,N$, which, roughly speaking,
correspond to the prelimit versions of the measures $\{\mu_0^i\}$
that appear in the theorem above. We now describe this sequence.

For $u\in \RR^m,\  v \in G, \ r \in (0, \infty)$, define, for $t\in
[0,1]$, \beq {\bf z}(u,v,r|t)&\equiv& {\bf z}(t)=v+(b(v)r +
\sigma(v)\sqrt{r}u)t,\\
{\bf x}(u,v,r|t)&\equiv& {\bf x}(t)=\Gamma({\bf z})(t),\\
{\bf y}(u,v,r|t)&\equiv& {\bf y}(t)={\bf x}(t)-{\bf z}(t).\eeq Then,
one can represent the trajectory ${\bf y}$ as \be \label{refproc}
{\bf y}(t)=\sum_{i=1}^N d_i \int_0^t \alpha_i(s)d|{\bf y}|(s); \quad
t\in [0,1],\ee where $ \alpha_i(s)\equiv  \alpha_i(u,v,r|s)\in
[0,1]$ and $ \alpha_i(s)>0$ only if ${\bf x}(s) \in F_i$. Also, let,
for $t\in [0,1]$ \bes {\bf \Pi^t}(u,v,r)={\bf z}(1)+t({\bf
x}(1)-{\bf z}(1)),\ees
 \bes {\bf L^i}(u,v,r)=\int_0^1 {\alpha_i}(t)d|{\bf y}|(t), \quad i=1,...,N.\ees
 Finally for $k \in \NN_0$, let
\bes \Pi_k^t= {\bf \Pi^t}(U_{k+1},X_k,\step_{k+1}),\; L_k^i=
{\bf L^i}(U_{k+1},X_k,\step_{k+1}).\ees

For $k \in \NN_0$ and $i=1,...,N$, define a $\clm_F(\RR^m)$ valued
random variable $m_k^i$ by the relation \be \label{mki}
 \langle \psi,m_k^i \rangle = \int_0^1 \EE_{X_k}[\psi(\Pi_k^t)L_k^i]dt,\quad \psi \in BM_+(\RR^m),\ee
where $\EE_{X}[Z]$ denotes $\EE[Z|X]$, and $BM_+(\RR^m)$ is the space of nonnegative bounded measurable functions on $\RR^m$.

For $n \in \NN$ and $i=1,...,N$, let $\mu_n^i$ be a $\clm_F(\RR^m)$
valued random variable defined as
$$\mu_n^i(A)=\frac{1}{\stepsum_n}\sum_{k=0}^{n-1}m_k^i(A); \quad A
\in \clb(\RR^m).$$

The following lemma relates the above family of random measures with
our approximation scheme. Recall the definition of the filtration
$\{\clf_k\}$ in Section \ref{sec:tight}.

\begin{lemma}\label{lemma:telesum}
For every $f \in C_b^2(\RR^m)$, there exists a
sequence of real random variables $\{\xi_n^f\}_{n\in\NN}$ such that
\be
\label{telesum} \frac{1}{\stepsum_n} \sum_{k=0}^{n-1}\EE[f(X_{k+1})
- f(X_k)|\clf_{k}]=\sum_{i=1}^N \mu_n^i (D_i f)+ \nu_n(\cla
f)+\xi_n^f, \ee and $\sup_n \xi_n^f(\omega)<\infty$ a.s. Furthermore if $f$ has compact support then $\xi_n^f \rightarrow 0$ a.s. as $n\rightarrow \infty$.
\end{lemma}
\begin{proof}
Fix $(u,v,r) \in \RR^m \times G \times (0, \infty)$. Using the
notation introduced above, we have from Taylor's theorem,
 \bes
    f(z(1))-f(v)=\langle \nabla f(v),\eta \rangle + \frac{1}{2}\eta' D^2 f(v)\eta +R_2(v,z(1))
    \ees
    where
  \bes R_2(x,y)=f(y)-f(x)-\langle \nabla f(x),
y-x\rangle-\frac{1}{2} (y-x)^T D^2f(x)(y-x)\ees
 and \bes \eta \equiv
\eta(u,v,r)=b(v)r + \sigma(v)\sqrt{r}u.\ees Define
\bes r_2(x,y)=
\frac{1}{2} \sup_{t\in (0,1)}||D^2f(x+t(y-x))-D^2f(x)||,
\ees then we
have $|R_2(x,y)| \le r_2(x,y)|x-y|^2$.

    Also
    \bes \begin{split}
    f({\bf x}(1))-f({\bf z}(1)) &=\int_0^1 \frac{df({\bf z}(1)+t({\bf x}(1)-{\bf z}(1)))}{dt} dt\\
    &=\int_0^1 \nabla{f({\bf z}(1)+t({\bf x}(1)-{\bf z}(1)))}\cdot ({\bf x}(1)-{\bf z}(1)) {dt}\\
    &=\sum_{i=1}^N \int_0^1 \nabla{f({\bf z}(1)+t({\bf x}(1)-{\bf z}(1)))}{dt}\cdot d_i \int_0^1 \alpha_i(t)d|{\bf y}|(t).
\end{split}
    \ees

    Fix a $k\in \NN$ and let $v=X_k$, $u=U_{k+1}$ and $r=\step_{k+1}$. Then
    \bes \EE[f(X_{k+1}) - f(X_k)|\clf_{k}] = \EE[f({\bf x}(1))-f({\bf z}(1)) + f({\bf z}(1))-f(v)|\clf_{k}].\ees

From the definition of $m_k^i$ in \eqref{mki} and observing that
$\{X_k\}$ is a Markov chain (with respect to the filtration $\{\clf_k\}$) and $U_{k+1}$ is independent of $\clf_k$, it follows that \bes \EE[
f({\bf x}(1))-f({\bf z}(1))|\clf_{k}] = \sum_{i=1}^N m_k^i(D_i
f),\ees Using independence of $U_{k+1}$ from $\clf_k$ once more, \bes
\begin{split}  \EE[ f({\bf z}(1))-f(v)|\clf_{k}] =&
\step_{k+1}\langle \nabla f(X_k),b(X_k) \rangle +
\frac{1}{2}\step_{k+1} \sigma(X_k)' D^2 f(X_k)\sigma(X_k)\\
&+\frac{1}{2}\step_{k+1}^2 b(X_k)' D^2 f(X_k)b(X_k)+
\EE[R_2(X_k,X_k+\eta_k)|\clf_{k}]\\
=&\step_{k+1}\cla f(X_k)+\xi^f(k),
\end{split} \ees
where \[\xi^f(k)=\frac{1}{2}\step_{k+1}^2 b(X_k)' D^2 f(X_k)b(X_k)+
\EE[R_2(X_k,X_k+\eta_k)|\clf_{k}]\] and $\eta_k=\eta(U_{k+1},X_k,
\step_{k+1})$.

Thus we have \bes \begin{split} \frac{1}{\stepsum_n}
\sum_{k=0}^{n-1}&\EE[f(X_{k+1}) - f(X_k)|\clf_{k}]\\
&=\frac{1}{\stepsum_n}\sum_{k=0}^{n-1} \left[\sum_{i=1}^N m_k^i(D_i
f)+ \step_{k+1}\cla f(X_k)+\xi^f(k)\right]\\
&=\sum_{i=1}^N \mu_n^i (D_i f)+ \nu_n(\cla
f)+\frac{1}{\stepsum_n}\sum_{k=0}^{n-1}\xi^f(k).
\end{split}\ees
Equality in \eqref{telesum} follows on taking $\xi_n^f=\frac{1}{\stepsum_n}\sum_{k=0}^{n-1}\xi^f(k)$.

We now show that $\sup_n \xi_n^f(\omega)<\infty$ a.s.
Write $$\xi^f(k)=\frac{1}{2}\step_{k+1}^2 b(X_k)' D^2
f(X_k)b(X_k)+ \EE[R_2(X_k,X_k+\eta_k)|\clf_{k}]\equiv
\xi_1^f(k)+\xi_2^f(k).$$ The term
$\frac{1}{\stepsum_n}\sum_{k=0}^{n-1}\xi_1^f(k)$ converges to zero
because of the boundedness of $b$ and $D^2 f$. Consider now the
contribution from $\xi_2^f(k)$. Let for $p \in \RR_+$,
$$h(p)=\frac{1}{2}\sup_{\begin{subarray}{l} x_1,x_2 \in \RR^m\\|x_1-x_2|\le p\end{subarray}}||D^2f(x_2)-D^2f(x_1)||.$$ Then $$|R_2(X_k,X_k+\eta_k)| \le h(\eta_k)|\eta_k|^2$$ and so for some $\kappa_1\in (0,\infty)$,
\bes \begin{split}
|\xi_2^f(k)|&\le \EE[h(\eta_k)|\eta_k|^2|\clf_k]\\
&\le||h||_\infty\kappa_1\step_{k+1}.
\end{split} \ees
Thus $\sup_n \xi_n^f(\omega)<\infty$ a.s. This completes the first part of the lemma.

Finally if $f$ in addition has compact support, we have $h(p)\rightarrow 0$ as $p\rightarrow 0$. Fix $\epsilon>0$. Since $b$, $\sigma$ are bounded, we can find for each $\theta \in (0,\infty)$, $k_\theta \in \NN$ such that for every $k\geq k_\theta$,
\bes  |h(b(x_k)\step_{k+1}+\sigma(x_k)\sqrt{\step_{k+1}}U_{k+1})|1_{|U_{k+1}|\leq\theta}\leq\epsilon.
\ees
Also, for some $l_\eta \in(0,\infty)$, for all $k \in \NN$,
\bes
  \EE[|\eta_k|^2 1_{|U_{k+1}|\geq\theta}|\clf_k]\leq
l_\eta(\step^{3/2}_k+\step_k\EE[|U_1|^21_{|U_1|\geq\theta}]) \quad \text{a.s.,}
\ees
\bes
  \EE[|\eta_k|^2|\clf_k]\leq l_\eta\step_k \quad \text{a.s.}
\ees
Choose $\theta_0\in (0,\infty)$ such that
$\EE[|U_1|^21_{|U_1|\geq\theta_0}]\leq\epsilon$. Then
\bes
  \begin{split}
    \frac{1}{\stepsum_{n}}\sum^{n-1}_{k=k_{\theta_0}}|\xi_2^f(k)|\leq\epsilon
l_\eta\frac{1}{\stepsum_n}\sum^{n-1}_{k=k_{\theta_0}}\step_k+||h||_\infty
l_\eta(\frac{1}{\stepsum_n}\sum_{k=k_{\theta_0}}^{n-1}\step_k^{3/2}+\frac{\epsilon}{\stepsum_n}\sum_{k=k_{\theta_0}}^{n-1}\step_k).
  \end{split}
\ees
Thus,
\bes
     \frac{1}{\stepsum_{n}}\sum^{n-1}_{k=0}|\xi_2^f(k)|\leq
\frac{1}{\stepsum_n}\sum^{k_{\theta_0}-1}_{k=0}\xi_2^f(k)+\epsilon
l_\eta(1+||h||_\infty)+||h||_\infty
l_\eta\frac{1}{\stepsum_n}\sum_{k=k_{\theta_0}}^{n-1}\step_k^{3/2}.
\ees Sending $n\rightarrow\infty$ and then $\epsilon\rightarrow 0$,
we now see that
$\frac{1}{\stepsum_{n}}\sum^{n-1}_{k=0}|\xi_2^f(k)|\rightarrow 0$ as
$n\rightarrow \infty$. The result follows.
\end{proof}

The following lemma shows that the left side of the expression in
\eqref{telesum} converges to 0 as $n\rightarrow \infty$.

\begin{lemma}\label{lemma:telecgc}
For every $f \in C_b^2(G)$, \bes \frac{1}{\stepsum_n}
\sum_{k=0}^{n-1}\EE[f(X_{k+1}) - f(X_k)|\clf_{k}]\rightarrow 0
\text{ a.s., \ \ as } n\rightarrow \infty.\ees
\end{lemma}
\begin{proof}
We can split the sum into two terms: \bes \begin{split}
\frac{1}{\stepsum_n}
\sum_{k=0}^{n-1}&\EE[f(X_{k+1}) - f(X_k)|\clf_{k}]\\
=&\frac{1}{\stepsum_n}\sum_{k=0}^{n-1}\left(\EE[f(X_{k+1})|\clf_{k}]
-f(X_{k+1})\right)+\frac{1}{\stepsum_n}\sum_{k=0}^{n-1}\left(f(X_{k+1})
-f(X_{k})\right)\\
=&T_1+T_2.
\end{split}\ees
Note that,
\bes
  |T_2|=\frac{1}{\stepsum_n}|f(X_{n}) -f(X_{0})| \rightarrow
  0,
\ees
as $n\rightarrow \infty$, since $f$ is bounded and
$\stepsum_n\rightarrow\infty$.
Also, using Kronecker's
Lemma,  \bes
T_1=\frac{1}{\stepsum_n}\sum_{k=0}^{n-1}\left(\EE[f(X_{k+1})|\clf_{k}]
-f(X_{k+1})\right) \ees
will converge to 0 once the martingale
\bes
  M_n^f:=\sum_{k=1}^{n-1}\frac{1}{\stepsum_k}(\EE[f(X_{k+1})|\clf_k]-f(X_{k+1}))
\ees
converges a.s.
Finally observing that $\EE(f(X_{k+1})|\clf_k)$ minimizes the $L^2$ distance from $f(X_{k+1})$ among $\clf_k$ measurable
square integrable random variables,
\bes
  \begin{split}
  \EE\langle M^f\rangle_\infty&=\sum_{k\geq 1}(\frac{1}{\stepsum_k})^2\EE\left(f(X_{k+1})-\EE(f(X_{k+1})|\clf_k)\right)^2\\
 &\leq\sum_{k\geq 1}(\frac{1}{\stepsum_k})^2\EE\left(f(X_{k+1})-f(X_k)\right)^2\\
 &\leq ||Df||_\infty\sum_{k\geq 1}(\frac{1}{\stepsum_k})^2\EE\left(X_{k+1}- X_k\right)^2\\
 &\leq \kappa_1 \sum_{k\geq 1} \frac{\step_{k+1}}{\stepsum_k^2}\\
 &<\infty
  \end{split}
\ees for some constant $\kappa_1$, where the last inequality follows
from the observation that for a positive sequence $\step_k$,
$\sum_{k\geq 1}\step_{k+1}/\stepsum_k^2 <\infty$. The lemma
follows.
\end{proof} 

Next we consider the limit of the first term on the right side of \eqref{telesum}.
We can regard $\mu_n^i$ to be a finite measure on the one point
compactificaion of $\RR^m$, denoted as $\bar{\RR}^m$. In order to show that
$\{\mu_n^i\}$ is a.s. a precompact sequence in $\clm_F(\bar{\RR}^m)$, it
suffices to show that $\mu_n^i(\RR^m)$ is an a.s. bounded sequence of $\RR_+$
valued random variables. This is shown in the following lemma.

\begin{lemma}\label{lemma:precmp}
For $i=1,...,N$, $$\sup_n \mu_n^i(\RR^m) < \infty, \quad \text{a.s.}$$
\end{lemma}

\begin{proof}
Let $g \in C^2_b (\RR^m)$ be as in Lemma \ref{lemma:g}. Then  for fixed $(u,v,r)\in \RR^m\times
G\times(0,\infty)$ and with notation as introduced above Lemma
\ref{lemma:telesum}, \be \label{equ:g}\begin{split} g({\bf
x}(1))=&g(v) + \int_0^1 [\nabla g ({\bf x}(s))\cdot (b(v)r +
\sigma(v)\sqrt{r}u)]ds+\sum_{i=1}^N \int_0^1 d_i \cdot\nabla g
({\bf x}(s)) \alpha_i(s)d|{\bf y}|(s)\\
\end{split}\ee
Since $\alpha_i(s)$ is nonzero only when ${\bf x}(s) \in F_i$, and $ \langle
\nabla g(x), d_i \rangle \ge 1$, for all  $x \in F_i$, $i \in
\{1,...,N\}$,
we
have \be\label{equ:g1}
\begin{split} \sum_{i=1}^N{\bf L^i}(v,u,r) &=
\sum_{i=1}^N
\int_0^1 \alpha_i(s)d|{\bf y}|(s)\\
&\le \sum_{i=1}^N \int_0^1 d_i \cdot\nabla g ({\bf x}(s))
\alpha_i(s)d|{\bf y}|(s)\\
&\le | g({\bf x}(1))-g(v)|+||\nabla g||_\infty|b(v)r + \sigma(v)\sqrt{r}u|\\
&\le ||\nabla g||_\infty|{\bf x}(1)-v|+||\nabla g||_\infty|b(v)r + \sigma(v)\sqrt{r}u|\\
&\le ||\nabla g||_\infty(K+1)|b(v)r + \sigma(v)\sqrt{r}u|,
\end{split}\ee
where the second inequality uses \eqref{equ:g}, and the last inequality uses the
Lipschitz property of the Skorokhod map.

Let $\kappa_1 = ||\nabla g||_\infty(K+1) a_1$, then from (\ref{equ:g1}) we have for $i\in
\{1,...,N\}$, \be \label{Lki} L_k^i\le \kappa_1 \left (\sqrt{\step_{k+1}}
|U_{k+1}| + \step_{k+1}\right ).
\ee Also note that,
\be\label{supxv} \sup_{t\in
[0,1]}|{x}^k(t)-X_k| \le K|b(X_k)\lambda_{k+1} + \sigma(X_k)\sqrt{\lambda_{k+1}}U_{k+1}|\le  K a_1
\sqrt{\lambda_{k+1}} |U_{k+1}| + K a_1 \lambda_{k+1},\ee
and for $t\in [0,1]$, \be \label{equ:g2} |\Pi_k^t-X_k| \le
t|x^k(1)-v|+(1-t)|z^k(1)-v|\le (K+1)a_1\step_{k+1}+
(K+1)a_1\sqrt{\step_{k+1}} |U_{k+1}|,\ee
where ${x}^k(t)= {\bf x}(U_{k+1},X_k,\step_{k+1}|t)$, ${z}^k(t)=
{\bf z}(U_{k+1},X_k,\step_{k+1}|t)$.
Combining \eqref{Lki}-\eqref{equ:g2} we have that
\be \label{ins451}
  \EE_{X_k}(|\Pi_k^t-x^k(s^i_k)|L_k^i)\leq
(2K+1)a_1\kappa_1 m\step_{k+1}+\varphi(\step_{k+1})\step_{k+1}, \ee
where $\varphi:(0,\infty)\rightarrow(0,\infty)$ is a bounded
function satisfying $\varphi(\alpha)\rightarrow 0$ as
$\alpha\rightarrow 0$.

Next note that $L_k^i$ is not equal to 0 only if there exists $s\in
[0,1]$ such that $\alpha_i^k(s)>0$, i.e., $x^k(s) \in F_i$,
where $\alpha_i^k(t)\equiv \alpha_i(U_{k+1},X_k,\step_{k+1}|t)$. And in
that case, \bes D_i g(\Pi_k^t) \ge D_i g(x^k(s))-||D^2 g||_\infty
|\Pi_k^t-x^k(s)|\geq 1-||D^2 g||_\infty|\Pi_k^t-x^k(s)|.\ees
Let $A^i_k=\{\omega: \text{there exists }s \in [0,1]\text{ such that }\alpha_i^k(s)>0\}$ and
\begin{subnumcases}{s^i_k(\omega)=}\nonumber
\inf\{s \in [0,1]: \alpha_i^k(s)>0\} & if $\omega \in A^i_k$,\\ \nonumber
1 & if $\omega \notin A^i_k$.
\end{subnumcases}
Then, from (\ref{ins451}),
 \bes \begin{split}
 \EE_{X_k}[D_i g(\Pi_k^t)L_k^i1_{A^i_k}] &\ge \EE_{X_k}[L_k^i1_{A^i_k}]-
 ||D^2 g||_\infty \EE_{X_k}[ |\Pi_k^t-x^k(s^i_k)|L_k^i1_{A^i_k}]\\
 &\ge \EE_{X_k}[L_k^i]-||D^2 g||_\infty ((2K+1)a_1\kappa_1 m\step_{k+1}+\varphi(\step_{k+1})\step_{k+1})
 \end{split} \ees
Thus we have
 \bes \begin{split}
 \langle D_i g, m_k^i \rangle &= \int_0^1 \EE_{X_k}[D_i g(\Pi_k^t)L_k^i]dt\\
 &=\int_0^1 \EE_{X_k}[D_i g(\Pi_k^t)L_k^i1_{A^i_k}]dt\\
 &\ge \langle 1,m^i_k\rangle-||D^2 g||_\infty\left ( (2K+1)a_1m\kappa_1\step_{k+1}+\varphi(\step_{k+1})\step_{k+1}\right).
 \end{split} \ees
Rearranging the terms, we have \bes
\langle 1, m_k^i \rangle \le \langle D_i g, m_k^i \rangle +||D^2
g||_\infty \left ( (2K+1)a_1m\kappa_1\step_{k+1} + \varphi(\step_{k+1})\step_{k+1}\right ).\ees
Summing over $k$ from 0 to $n-1$ and $i$ from 1 to $N$, we obtain
\be \label{equ:g3}
  \sum_{i=1}^N\langle 1,\mu^i_n\rangle\le\sum_{i=1}^N\langle
D_ig,\mu^i_n\rangle+N||D^2 g||_\infty\left ((2K+1)a_1\kappa_1 m+|\varphi|_\infty \right).
\ee
Using Lemma \ref{lemma:telesum}
\bes
  \sum_{i=1}^N\mu_n^i(D_ig)=\frac{1}{\stepsum_n}\sum_{k=0}^{n-1}\EE[g(X_{k+1})-g(X_k)|\clf_k]-\nu_n(\cla
g)-\xi_n^g. \ees
Since $g\in C_b^2(\RR^m)$, the second term on the right
side is bounded. Also from Lemma \ref{lemma:telecgc}, the first term converges to 0 as $n\rightarrow \infty$. Finally from Lemma \ref{lemma:telesum}, the third term is bounded, a.s.

From this it follows that
\bes
  \sup_n\sum_{i=1}^N\mu_n^i(D_ig)<\infty \quad \text{a.s.}
\ees
Result follows on using this observation in \eqref{equ:g3}.
\end{proof}

The following lemma will be used to show that for a.e. $\omega$, any limit point
of $\mu_n^i(\omega)$ is supported on $F_i$, $i=1,...,N$.

\begin{lemma}\label{lemma:face}
Fix $i \in \{1,...,N\}$. Let $\psi \in C_c^2(\RR^m)$ be such that
$\psi(x) \ge 0$ for all $x \in \RR^m$. Suppose that there is a $\epsilon >0$,
such that $\psi(x) =0$ if $\text{dist}(x, F_i) \le \epsilon$. Then
$$ \int \psi(x)\mu_n^i(dx) \rightarrow 0, \text{ a.s. as } n\rightarrow \infty.$$
\end{lemma}
\begin{proof}
We have \be \label{equ:g4} \begin{split} \langle \psi, \mu_n^i \rangle
&=\frac{1}{\stepsum_n} \sum_{k=0}^{n-1}  \langle \psi, m_k^i \rangle\\
&=\frac{1}{\stepsum_n}\sum_{k=0}^{n-1} \int_0^1 \EE_{X_k}[\psi
(\Pi_k^t)L_k^i]dt\\
&=\frac{1}{\stepsum_n}\sum_{k=0}^{n-1} \int_0^1 \EE_{X_k}\int_0^1 \psi
(\Pi_k^t) \alpha_i^k(s) d|y^k|_s dt\\
&\le\frac{|\psi|_\infty}{\stepsum_n}\sum_{k=0}^{n-1} \EE_{X_k}
\int_0^1 \int_0^1 1_{\{|\Pi_k^t-{x}^k(s)|>\epsilon\}} \alpha_i^k(s) d|y^k|_s dt,
\end{split} \ee
where, recall that $\alpha_i^k(s)\equiv \alpha_i(U_{k+1},X_k,\step_{k+1}|s)$ and
${x}^k(s)= {\bf x}(U_{k+1},X_k,\step_{k+1}|s)$, ${y}^k(s)=
{\bf y}(U_{k+1},X_k,\step_{k+1}|s)$. The last inequality in the
above display follows from noting that $\alpha_i^k(s)>0$ only when
${x}^k(s)\in F_i$ and if for such a $s$, $|\Pi_k^t-{x}^k(t)|\le\epsilon$,we have by our choice of $\psi$ that $\psi(\Pi_k^t)=0$.

Next note that
\bes
  \{(t,s,\omega):|\Pi_k^t-{x}^k(s)|>\epsilon\}\subset
\{(t,s,\omega):|{x}^k(1)-{x}^k(s)|>\epsilon\}\cup\{(t,s,\omega):|{z}^k(1)-{x}^k(s)|>\epsilon\}, \ees where recall that ${z}^k(t)={\bf z}(U_{k+1},{X}_k,\step_{k+1}|t)$.

Also, from the Lipschitz property of the Skorokhod map, \bes |{x}^k(1)-{x}^k(s)| \le
Ka_1\step_{k+1}+Ka_1\sqrt{\step_{k+1}}|U_{k+1}|,\ees and \bes |{z}^k(1)-{x}^k(s)| \le |{z}^k(1)-X_k|+|X_k-{x}^k(s)| \le
(K+1)a_1\step_{k+1}+(K+1)a_1\sqrt{\step_{k+1}}|U_{k+1}|.\ees Thus
\[\{\omega:|\Pi_k^t-{x}^k(s)|>\epsilon \text{ for some }
t,s\in[0,1]\} \subset \{\omega:|U_{k+1}(\omega)|\ge p_k\},\] where
$p_k=\frac{\epsilon/((K+1)a_1)-\step_{k+1}}{\sqrt{\step_{k+1}}}$.
Using this observation in \eqref{equ:g4}, we have
 \bes
\begin{split} \langle \psi, \mu_n^i \rangle
&\le\frac{|\psi|_\infty}{\stepsum_n}\sum_{k=0}^{n-1} \EE_{X_k} \int_0^1
\int_0^1 1_{\{|U| \ge p_k\}} \alpha_i^k(s) d|y^k|_s dt\\
&\le\frac{|\psi|_\infty}{\stepsum_n}\sum_{k=0}^{n-1} \EE_{X_k}\left ( 1_{\{|U| \ge
p_k\}}\int_0^1  \alpha_i^k(s) d|y^k|_s \right )\\
&\le\frac{|\psi|_\infty}{\stepsum_n}\sum_{k=0}^{n-1}
\sqrt{\left(\EE_{X_k}\left(\int_0^1  \alpha_i^k(s) d|y^k|_s\right)^2\right)\PP (|U| \ge p_k)}.\\
\end{split} \ees
From  \eqref{equ:g1} it follows that for some
$\kappa_1\in (0,\infty)$, $\sup_k\EE_{X_k}\left(\int_0^1  \alpha_i^k(s) d|y^k|_s\right)^2\le \kappa_1$.
Also using Condition \ref{expint}, $\EE|U|^j<\infty \text{ for all }
j\ge 1$. Choose $k_0$ large enough so that $\step_{k+1}\le
\frac{\epsilon}{2(K+1)a_1}\text{ for all } k\ge k_0$.

Fix $j>4$, then
\bes
    \langle \psi,\mu^i_n\rangle \le
\frac{|\psi|_\infty}{\stepsum_n}\sqrt{\kappa_1}k_0+\frac{|\psi|_\infty}{\stepsum_n}\sum_{k=k_0}^{n-1}\sqrt{\kappa_1}(\EE|U|^j)^{1/2}p_k^{-j/2}.
\ees
The result now follows on observing that for some $\kappa_2\in (0,\infty)$,
$p_k^{-j/2}\le \kappa_2\step_{k+1}^{j/4}$ for all $k\ge k_0$.
\end{proof}


We are now ready to complete the proof of Theorem \ref{thm:cgce}.

\begin{proof}[Proof of Theorem \ref{thm:cgce}]
Fix $f \in C_c^2(G)$. Then such a function can be extended to a function in $C_c^2(\RR^m)$. We denote this function once more by $f$. Then from Lemma \ref{lemma:telesum}, \be
\label{telesum2} \frac{1}{\stepsum_n} \sum_{k=0}^{n-1}\EE[f(X_{k+1})
- f(X_k)|\clf_{k}]=\sum_{i=1}^N \mu_n^i (D_i f)+ \nu_n(\cla
f)+\xi_n^f. \ee From Lemmas
\ref{lemma:tight}, \ref{lemma:telesum}, \ref{lemma:telecgc} and
\ref{lemma:precmp}, there exists $\Omega_0 \in \clf$ such
that $\PP(\Omega_0)=1$ and for every $\omega \in \Omega_0$,
\begin{itemize}
 \item $\{\nu_n(\omega)\}_n$ is precompact in $\clp(G)$,
 \item $\{\mu^i_n(\omega)\}_n$ is precompact in $\clm_F(\bar{\RR}^m)$, for
every $i=1,...,N$,
 \item Left hand side of \eqref{telesum2} converges to 0,
 \item $\xi_n^f(\omega)$ converges to 0.
\end{itemize}
Fix a $\omega \in \Omega_0$ and let $\nu_\infty(\omega)$,
$\mu^i_\infty(\omega)$, $i=1,...,N$, be a subsequential limit of
$\nu_n(\omega)$ and $\mu^i_n(\omega)$, respectively. Then from
\eqref{telesum2} and the above observations, we have ( suppressing
$\omega$ ) \bes
    \nu_\infty(\cla f) +\sum_{i=1}^N \mu_\infty^i(D_i f)=0.
    \ees

To complete the proof, in view of Theorem \ref{thm:inv}, it suffices
to argue that \be \label{face} \int_{\RR^m}
1_{F_i^c}(x)\mu_\infty^i(\omega)(dx)=0. \ee By convergence of $\mu_n^i$ to
$\mu_\infty^i$, we have   for
every $\psi$ as in Lemma \ref{lemma:face},
$$ \int_{\RR^m} \psi(x)\mu_\infty^i(\omega)(dx)=0.$$
Therefore
\bes
  \int_{\RR^m}1_{F_i^{\epsilon,r}}(x)\mu_\infty^i(dx)=0 \quad \forall \epsilon,r >0,
\ees where $F_i^{\epsilon,r}=\{x\in \RR^m| \text{ dist}(x,F_i)\ge
\epsilon \text{ and } |x|\le r\}. $ The equality in \eqref{face} now
follows on sending $\epsilon\rightarrow 0$ and $r\rightarrow
\infty$.
\end{proof}
\subsection{Proof of Theorem \ref{thm:unbddmom}}\label{sec:thmunbdd}
%
Recall $c$ from Lemma \ref{propT} and $\varpi$ from (\ref{ins204}).  Fix $\zeta \in (0, \varpi c)$.  We will prove the theorem with such a choice of $\zeta$.
Consider an $f$ as in the statement of the theorem.  Then there exists constant $\kappa_1$ such that $|f(x)|\le \kappa_1 e^{\zeta |x|}$
Without loss of generality, we assume $f \ge 0$.

From Theorem \ref{thm:cgce}, for any $L >0$, we have
$$\int(f\wedge L)d\nu_n \rightarrow \int(f\wedge L)d\nu\quad \text{a.s.}$$
In order to prove the theorem, it suffices to show that
\[\int(f\wedge L)d\nu_n \rightarrow \int f d\nu_n, \mbox{ and } \int(f\wedge L)d\nu \rightarrow \int fd\nu, \mbox{ as } L\rightarrow \infty.\]
First, consider
\begin{equation*}
\begin{split}
\sup_n\left[\int f d\nu_n-\int(f\wedge L)d\nu_n\right]&\le \sup_n \int 1_{f>L}f d\nu_n\\
&\le  \sup_n\left(\nu_n^{1/p}(f>L)[\nu_n(f^q)]^{1/q}\right),
\end{split}
\end{equation*}
where $p,q \in (1, \infty)$ are such that $p^{-1} + q^{-1} = 1$ and the last inequality follows from H\"older's inequality.
Choose $q > 1$ such that $\zeta q <\varpi c$, then from Lemma \ref{lemma:tight} we have
\be \label{ins945}\sup_n[\nu_n(f^q)]^{1/q}\le \kappa_1\sup_n[\int e^{\zeta q|x|}\nu_n(dx)]^{1/q} \le \kappa_1\sup_n \nu_n^{1/q}(V) < \infty, \quad \text{a.s.}\ee
Using Markov's Inequality, we have
$$\nu_n^{1/p}(f>L)\le\frac{\nu_n^{1/p}(f)}{L^{1/p}},$$
which using (\ref{ins945}) converges to 0 as $L$ goes to infinity. Combining the above three displays, we have
\be \label{ins947} \sup_n\left[\int f d\nu_n-\int(f\wedge L)d\nu_n\right]\rightarrow 0, \quad\text{a.s.  as } L \rightarrow \infty.\ee

Also, from Fatou's lemma we have, for a.e. $\omega$,
\begin{equation*}
\begin{split}
\int f d\nu-\int(f\wedge L)d\nu&=\int (f -f\wedge L)d\nu\\
&\le \liminf_n\int (f -f\wedge L)d\nu_n\\
&\le \sup_n\int (f -f\wedge L)d\nu_n.
\end{split}
\end{equation*}
Using (\ref{ins947}) the last expression converges to $0$ as $L \to \infty$.  The result follows.

\section{Proof of Theorem \ref{thm:rate}}\label{ins2107}
We begin with a few preliminary lemmas.

\begin{lemma} \label{lemma:decomp}If $\phi \in C^2(G)$, then
\bes \begin{split} \stepsum_n \nu_n(\cla \phi) = \sum_{k=1}^n
\step_k \cla \phi(X_{k-1})= Z_n^{(0)}-(N_n + \sum_{i=1}^4
Z_n^{(i)}+\sum_{i=1}^4 Y_n^{(i)})
\end{split}\ees
with
\beq Z_n^{(0)}&=&\phi(X_n)-\phi(X_0),\\
N_n&=&\sum_{k=1}^n \sqrt{\step_k} \langle \nabla \phi(X_{k-1}),
\sigma(X_{k-1}) U_k\rangle,\\
Z_n^{(1)}&=&\frac{1}{2} \sum_{k=1}^n \step_k^2 b(X_{k-1})^T
D^2\phi(X_{k-1})b(X_{k-1}),\\
Z_n^{(2)}&=&\sum_{k=1}^n \step_k^{3/2} b(X_{k-1})^T
D^2\phi(X_{k-1})\sigma(X_{k-1}) U_k,\\
Z_n^{(3)}&=&\frac{1}{2} \sum_{k=1}^n \step_k [(\sigma(X_{k-1})
U_k)^T D^2\phi(X_{k-1})(\sigma(X_{k-1}) U_k)\\
&&-\EE((\sigma(X_{k-1})
U_k)^T D^2\phi(X_{k-1})(\sigma(X_{k-1}) U_k)|\clf_{k-1})],\\
Z_n^{(4)}&=&\sum_{k=1}^n R_2(X_{k-1},X_{k}),\eeq
and
\beq
Y_n^{(1)}&=&\sum_{k=1}^n  \langle \nabla \phi(X_{k-1}),
y_{k-1}\rangle,\\
Y_n^{(2)}&=&\frac{1}{2} \sum_{k=1}^n  y_{k-1}^T
D^2\phi(X_{k-1})y_{k-1},\\
Y_n^{(3)}&=&\sum_{k=1}^n  \step_k b(X_{k-1})^T
D^2\phi(X_{k-1})y_{k-1},\\
Y_n^{(4)}&=&\sum_{k=1}^n  \sqrt{\step_k} y_{k-1}^T
D^2\phi(X_{k-1})\sigma(X_{k-1}) U_k, \eeq where
$R_2(x,y)=\phi(y)-\phi(x)-\langle \nabla \phi(x),
y-x\rangle-\frac{1}{2} (y-x)^T D^2\phi(x)(y-x)$, and $y_k={\bf
y}(U_{k+1},X_k,\step_{k+1}|1)$.
\end{lemma}

\begin{proof} Denote $\delta \phi(X_k)= \phi(X_k)- \phi(X_{k-1})$
and $\delta X_k=X_k-X_{k-1}$. We deduce from \eqref{scheme2012} that
\bes
\begin{split}
\delta\phi(X_k)=&\langle \nabla\phi(X_{k-1}),\delta
X_k\rangle+\frac{1}{2} \delta X_k^T D^2\phi(X_{k-1})\delta
X_k+R_2(X_{k-1},X_{k})\\
=&\langle \nabla \phi(X_{k-1}), y_{k-1}\rangle + \step_k \cla
\phi(X_{k-1}) + \sqrt{\step_k} \langle \nabla \phi(X_{k-1}),
\sigma(X_{k-1}) U_k\rangle\\
 &+ \frac{1}{2}  y_{k-1}^T
D^2\phi(X_{k-1})y_{k-1} + \frac{1}{2}  \step_k^2 b(X_{k-1})^T
D^2\phi(X_{k-1})b(X_{k-1})\\
&+\frac{1}{2} \step_k [(\sigma(X_{k-1}) U_k)^T
D^2\phi(X_{k-1})(\sigma(X_{k-1}) U_k)-\EE((\sigma(X_{k-1})U_k)^T
D^2\phi(X_{k-1})(\sigma(X_{k-1}) U_k)|\clf_{k-1})]\\&+\step_k
b(X_{k-1})^T D^2\phi(X_{k-1})y_{k-1}+ \step_k^{3/2} b(X_{k-1})^T
D^2\phi(X_{k-1})\sigma(X_{k-1}) U_k\\&+ \sqrt{\step_k} y_{k-1}^T
D^2\phi(X_{k-1})\sigma(X_{k-1}) U_k+
R_2(X_{k-1},X_{k}).\\
\end{split}
\ees The lemma follows by summing the above equality over $k=1,...,n$
and rearranging the terms.
\end{proof}

\begin{lemma} \label{lemma:normal}
Let $W: G \rightarrow \RR$ be a continuous function such that
$\sup_{n \in \NN} \nu_n(W) <\infty$, a.s. Let $\phi \in C^1(G)$, be
such that $\lim_{|x| \rightarrow \infty}
 |\nabla \phi(x)|^2/ W(x)=0$. Then \bes
\frac{1}{\sqrt{\stepsum_n}}\sum_{k=1}^n \sqrt{\step_k} \langle
\nabla \phi(X_{k-1}), \sigma(X_{k-1}) U_k\rangle
\xrightarrow[]{\cll} \cln\left(0, \int_G |\sigma^T\nabla
\phi|^2d\nu\right).\ees
\end{lemma}

\begin{proof}This lemma follows from Theorem \ref{thm:cgce} using the martingale central limit theorem, along the lines of Proposition 2 of
\cite{LambertonPage02}. Details are left to the reader.
\end{proof}

\begin{lemma} \label{lemma:Zn4}Under the assumptions of Theorem
\ref{thm:rate}(b), we have, \bes
\frac{Z_n^{(4)}}{\stepsum_n^{(3/2)}}\xrightarrow[]{\PP}
\frac{1}{6}\int_G\int_{\RR^m}D^3\phi(x)(\sigma(x)u)^{\otimes 3}
\mu(du)\nu(dx),\ees as $n\rightarrow \infty$.
\end{lemma}

\begin{proof}
The proof is similar to that of Lemma 10 of \cite{LambertonPage02}
except for the treatment of reflection terms. Using the notation above
Theorem \ref{thm:rate} and in Lemma \ref{lemma:decomp}, we have \be
\label{Rr}R_2(x,y)=\frac{1}{6}D^3\phi(x)(y-x)^{\otimes
3}+R_4(x,y),\ee with
$$|R_4(x,y)|\le \frac{L}{6}|y-x|^4,$$
where $L$ is the Lipschitz constant for $D^3\phi$. Hence \be
\label{R2rk}R_2(X_{k-1},X_k)=\frac{1}{6}D^3\phi(X_{k-1})(\delta
X_k)^{\otimes 3} +r_k,\ee with
$$|r_k| \le \frac{L}{6}|\delta X_k|^4\le \kappa_1 \step_k^2 (1+|U_k|^4), \quad k \in \NN,$$
for some $\kappa_1 \in (0,\infty)$. Since $\EE |U_k|^4 :=\mu_4 <
\infty$ from Condition \ref{expint}, we have
\[\EE \sum_{k=1}^n |r_k|
\le\kappa_1(1+\mu_4)\sum_{k=1}^n\step_k^2.\] From the assumption
$\lim_{n \rightarrow \infty}(1/\sqrt{\stepsum_n})\sum_{k=1}^n
\step_k^{3/2}=\tilde{\step}\in (0,+\infty]$, we deduce that $\lim_{n
\rightarrow \infty}\sum_{k=1}^n \step_k^{3/2}=+\infty$ and \be
\label{step2}\lim_{n \rightarrow \infty}\sum_{k=1}^n \step_k^2
/\stepsum_n^{(3/2)}=0. \ee Therefore, \be
\label{rlim}\frac{1}{\stepsum_n^{(3/2)}}\sum_{k=1}^n r_k
\xrightarrow[]{L^1}0.\ee

Now consider the first term on the right side of \eqref{R2rk}. \bes
\begin{split} D^3\phi(X_{k-1})(\delta X_k)^{\otimes
3}=&D^3\phi(X_{k-1})(\step_k
b(X_{k-1})+\sqrt{\step_k}\sigma(X_{k-1})U_k + y_{k-1})^{\otimes 3}\\
=&\step_k^{3/2}D^3\phi(X_{k-1})(\sqrt{\step_k}
b(X_{k-1})+\sigma(X_{k-1})U_k)^{\otimes 3} + f_k^{(1)}(X_{k-1},U_k)\\
=&\step_k^{3/2}D^3\phi(X_{k-1})(\sigma(X_{k-1})U_k)^{\otimes 3} +
f_k^{(2)}(X_{k-1},U_k)+ f_k^{(1)}(X_{k-1},U_k),
\end{split}\ees
where $f_k^{(1)}(X_{k-1},U_k)$ and $f_k^{(2)} (X_{k-1},U_k)$ are defined
through the second and third equalities, respectively.

Next observe that
\begin{itemize}
\item From the assumptions, we have $$|\step_k
b(X_{k-1})+\sqrt{\step_k}\sigma(X_{k-1})U_k|\le
a_1\sqrt{\step_k}(|U_k|+\sqrt{\step_0}).$$
\item From \eqref{refproc} and \eqref{Lki}, we have $y_{k-1}=\sum_{i=1}^N d_i L_{k-1}^i$ and
for some $\kappa_2 \in (0, \infty)$,
$$L_{k-1}^i \le \kappa_2\sqrt{\step_k}(|U_k|+1), \mbox{ for all } k \in \NN$$
\item The term $L_{k-1}^i$ is non zero only if there exists $s\in
[0,1]$ such that $x_{k-1}(s) \in F_i$, where ${x}_{k-1}(s)={\bf
x}(U_k, X_{k-1},\step_k|s)$. And in that case, we have from
\eqref{supxv}, the Lipschitz property of $D^3\phi$ and \eqref{DD2D3}
that, for some $\kappa_3 \in (0, \infty)$,
$$|D^3_{\cdot jk}\phi(X_{k-1}) \cdot d_i|\le
\kappa_3\sqrt{\step_k}(|U_k|+1), \quad\forall j,k.$$
\end{itemize}
Combining these estimates, we see that $\EE \sum_{k=1}^n
|f_k^{(1)}(X_{k-1},U_k)| \le\kappa_4\sum_{k=1}^n\step_k^2$. Using
\eqref{step2} we now have \be
\label{ylim}\frac{1}{\stepsum_n^{(3/2)}}\sum_{k=1}^n
f_k^{(1)}(X_{k-1},U_k) \xrightarrow[]{L^1}0.\ee

For the term $f_k^{(2)} (X_{k-1},U_k)$, using the boundedness of
$D^3\phi$, $b$, and $\sigma$, it can be easily checked that $\EE
|f_k^{(2)}(X_{k-1},U_k)| \le\kappa_5\step_k^2$. Thus
\[\EE \sum_{k=1}^n
|f_b(X_{k-1},U_k)| \le\kappa_5\sum_{k=1}^n\step_k^2,\] and so using
\eqref{step2} once again, we have \be
\label{blim}\frac{1}{\stepsum_n^{(3/2)}}\sum_{k=1}^n
f_b(X_{k-1},U_k) \xrightarrow[]{\PP}0.\ee

Let $\Theta(X_{k-1},U_k)=
D^3\phi(X_{k-1})(\sigma(X_{k-1})U_k)^{\otimes 3}$. Since $\sup_k \EE
|\Theta(X_{k-1},U_k)|^2 < \infty$ and  $\lim_{n
\rightarrow \infty} \stepsum_n^{(3)}/(\stepsum_n^{(3/2)})^2=0$, we
have \be \label{conlim}\frac{1}{\stepsum_n^{(3/2)}}\sum_{k=1}^n
\step_k^{3/2}[\Theta(X_{k-1},U_k)
-\EE(\Theta(X_{k-1},U_k)|\clf_{k-1})]\xrightarrow[]{L^2}0.\ee
Observe that $\EE(\Theta(X_{k-1},U_k)|\clf_{k-1})=J(X_{k-1})$, where
$J$ is given by $$J(x):=\int_{\RR^m}D^3\phi(x)(\sigma(x)u)^{\otimes
3} \mu(du).$$ Since $\stepsum_n^{(3/2)}\rightarrow \infty$ as $n\rightarrow \infty$, we can apply Theorem \ref{thm:cgce} to the measure
$\tilde{\nu}_n=\frac{1}{\stepsum_n^{(3/2)}}\sum_{k=1}^n
\step_k^{3/2}\delta_{X_{k-1}}$. Since $J$ is continuous and bounded,
we have $\lim_{n\rightarrow \infty}\tilde{\nu}_n(J)=\int J d\nu$
a.s., and the lemma follows on combining this fact with
\eqref{Rr}-\eqref{conlim}.
\end{proof}

We are now ready to prove Theorem \ref{thm:rate}.

{\bf Proof of Theorem \ref{thm:rate}}  The proof is similar as the proof of Theorem 9 of
\cite{LambertonPage02}, once again the main difference is in the
treatment of reflection terms. Using the notation of Lemma
\ref{lemma:decomp}, we first observe that, for any sequence of
positive numbers $\{a_n\}_{n \in \NN}$ such that $\lim_{n\rightarrow
\infty} a_n = \infty$, we have $Z_n^{(0)}/a_n\rightarrow0$ in
probability. This is because, from Lemma \ref{Lemma:bnd}, the
sequence $\{X_n\}_{n \in \NN}$ is tight, and consequently so is
$\{\phi(X_n)\}_{n \in \NN}$ as well.

We also derive from the definitions of $Z_n^{(1)}$, $Z_n^{(2)}$ and
$Z_n^{(3)}$ the inequalities \be \EE|Z_n^{(1)}| \le \kappa_1
\sum_{k=1}^n \step_k^2 ||D^2\phi||_\infty,\label{Zn1}\ee and \be
\EE|Z_n^{(i)}|^2 \le \kappa_1\sum_{k=1}^n  \step_k^2
||D^2\phi||^2_\infty,\quad i=2,3, \label{Zn23}\ee
for some $\kappa_1 \in (0,\infty)$, for all $n \ge 1$.

(a) Now assume that $\lim_{n \rightarrow
\infty}(1/\sqrt{\stepsum_n}) \stepsum_n^{(3/2)}=0$. We then have
$\lim_{n \rightarrow \infty}\sum_{k=1}^n
\step_k^{2}/\sqrt{\stepsum_n}=0$, and it follows from \eqref{Zn1}
that $Z_n^{(1)}/\sqrt{\stepsum_n}\xrightarrow{L^1}0$. We also deduce
from \eqref{Zn23}, that
$Z_n^{(j)}/\sqrt{\stepsum_n}\xrightarrow{L^2}0$, for $j=2,3$.
Consider now $Z_n^{(4)}$. Denoting the Lipschitz norm of $D^2\phi$
by $L$, we have \bes |R_2(X_{k-1},X_{k})|\le
\frac{L}{2}|\Delta X_{k}|^3 \le \frac{L}{2}
a^3_1K^3(\step_{k}+ \sqrt{\step_{k}}|U_{k}|)^3,\ees where the second
inequality follows from the Lipschitz property of the Skorokhod map
(Condition \ref{skolip}). Thus, there exists  $\kappa_2 \in (0, \infty)$
such that, for all $n \ge 1$, \be \EE|Z_n^{(4)}| \le
\kappa_2\sum_{k=1}^n \step_k^{3/2}, \label{Zn4} \ee and therefore
$Z_n^{(4)}/\sqrt{\stepsum_n}\xrightarrow{L^1}0$.

We now, consider $Y_n^{(j)}$, for $j=1,2,3,4$. \bes Y_n^{(1)} =
\sum_{k=1}^n \langle \nabla \phi(X_{k-1}), y_{k-1}\rangle =
\sum_{k=1}^n D_i\phi(X_{k-1})L_{k-1}^i. \ees From \eqref{Lki}, we
have $|L_{k-1}^i| \le \kappa_3\sqrt{\step_k}(|U_k|+1)$. Also, for
any fixed $i$, $L_{k-1}^i$ is not equal to 0 only if there exists $x
\in F_i$, such that $||X_{k-1}-x|| \le a_1K\step_{k}+
a_1K\sqrt{\step_{k}} |U_{k}|$; and in that case, using Taylor's
theorem and the Lipschitz property of $D^2\phi$, there exists
 $\kappa_4 \in (0, \infty)$, such that, \bes
|D_i\phi(X_{k-1})-D_i\phi(x)-(X_{k-1}-x)^TD^2\phi(x) d_i |\le
\kappa_4 ||X_{k-1}-x||^2. \ees Combining this with \eqref{DD2}, we
have \bes |D_i\phi(X_{k-1})|\le \kappa_4 ||X_{k-1}-x||^2. \ees Thus
we have \be \EE|Y_n^{(1)}| \le \kappa_5\sum_{k=1}^n \step_k^{3/2},
\label{Yn1} \ee for some constant $\kappa_5$. Using similar
arguments as above, we obtain: \be \EE|Y_n^{(j)}| \le \kappa_5
\sum_{k=1}^n \step_k^{3/2}, \quad j =2,3,4.\label{Yn234}\ee Thus we
have that $Y_n^{(j)}/\sqrt{\stepsum_n}\xrightarrow{L^1}0$, for
$j=1,2,3,4$.

From Lemma \ref{lemma:tight}, and recalling the definition of $V$
(see \eqref{V}) we have that, for every $\zeta \in (0,c\varpi)$,
$$\sup_{n \in \NN} \int_G e^{\zeta|x|}\nu_n(dx) < \infty, \quad
\text{a.s.}$$ For such a $\zeta$, under the assumption that $\lim_{|x| \rightarrow
\infty}e^{-\zeta |x|}|\nabla \phi(x)|^2=0$, applying Lemma
\ref{lemma:normal}, we now have \bes \frac{N_n}{\sqrt{\stepsum_n}}
\xrightarrow[]{\cll} \cln\left(0, \int_G |\sigma^T\nabla
\phi|^2d\nu\right).\ees This completes the proof of part (a).

(b) Assume now that $\lim_{n \rightarrow
\infty}(1/\sqrt{\stepsum_n})\stepsum_n^{(3/2)}=\tilde{\step}\in
(0,+\infty]$. We then have that
\[\lim_{n \rightarrow
\infty}\stepsum_n^{(3/2)}=+\infty \mbox{ and } \lim_{n \rightarrow
\infty}\sum_{k=1}^n \step_k^2 /\stepsum_n^{(3/2)}=0.\] As before,
$Z_n^{(0)}/\stepsum_n^{(3/2)}\xrightarrow[]{\PP}0$. It follows from
\eqref{Zn1} that $Z_n^{(1)}/\stepsum_n^{(3/2)}\xrightarrow{L^1}0$,
and from \eqref{Zn23} that
$Z_n^{(j)}/\stepsum_n^{(3/2)}\xrightarrow{L^2}0$, for $j=2,3$.

Under the assumptions of part (b) (i.e. that $D^3\phi$ is bounded,
Lipschitz and \eqref{DD2D3} holds), we have, using similar arguments
as in part (a), for some $\kappa_6 \in (0,\infty)$, \be
\EE|Y_n^{(j)}| \le \kappa_6 \sum_{k=1}^n \step_k^{2}, \quad
j=1,...,4; \ n\ge1.\label{Yn1234}\ee Thus we have that
$Y_n^{(j)}/\stepsum_n^{(3/2)}\xrightarrow{L^1}0$, for $j=1,2,3,4$.

Applying Lemma \ref{lemma:normal} once again, we have, for $\phi$
satisfying $\lim_{|x| \rightarrow \infty}e^{-\zeta |x|}|\nabla
\phi(x)|^2=0$, \be \label{slown} \frac{N_n}{\sqrt{\stepsum_n}}
\xrightarrow[]{\cll} \cln\left(0, \int_G |\sigma^T\nabla
\phi|^2d\nu\right).\ee

Also from Lemma \ref{lemma:Zn4}  \be\label{slowzn42}
\frac{Z_n^{(4)}}{\stepsum_n^{(3/2)}}\xrightarrow[]{\PP}
\frac{1}{6}\int_G\int_{\RR^m}D^3\phi(x)(\sigma(x)u)^{\otimes 3}
\mu(du)\nu(dx)=-\tilde{m}.\ee

Now, if $\tilde{\step}<+\infty$, we have from the above observations
that $Z_n^{(j)}/\sqrt{\stepsum_n}\xrightarrow{\PP}0$, for
$j=0,1,2,3$, $Y_n^{(j)}/\sqrt{\stepsum_n}\xrightarrow{\PP}0$, for
$j=1,2,3,4$ and \be \label{slowzn41}
\frac{Z_n^{(4)}}{\sqrt{\stepsum_n}}\xrightarrow{\PP}-\tilde{\step}\tilde{m}.
\ee The statement in \eqref{slow1} now follows on combining this
with \eqref{slown}.

Finally, if $\tilde{\step}=+\infty$, we have
$Z_n^{(j)}/\stepsum_n^{(3/2)}\xrightarrow{\PP}0$, for $j=0,1,2,3$,
$Y_n^{(j)}/\stepsum_n^{(3/2)}\xrightarrow{\PP}0$, for $j=1,2,3,4$
and $N_n/\stepsum_n^{(3/2)}\xrightarrow{\PP}0$, and \eqref{slow2}
follows from \eqref{slowzn42}. This completes the proof of Theorem
\ref{thm:rate}.
$\hfill \Box$

\section{Numerical Results}\label{Sec:sim}

%
%
%
%

\subsection{Evaluation of the Euler Time Step.} \label{Sec:eval}
A key step in
simulating the sequence $\{X_k\}$ in \eqref{scheme2012} is the
evaluation of $\cls(X_k, Y_{k+1}-X_k)$, where $\cls:G \times \RR^m
\rightarrow G$ is the time-1 Skorokhod map defined in
\eqref{ins2000}. In this section we describe a procedure for
computing $\cls(x,v)$ that uses well known relationships between
Skorokhod problems and linear complementary problems (LCP). We
restrict ourselves to a setting where $N=m$ and $G=\RR^m_+$. We
begin by recalling the basic formulation of the LCP (see
\cite{CoPaSt09}). For $j \in \NN$, a $j \times j$ matrix $R$ and a
$j$-dimensional vector $\theta$, the LCP for $(R, \theta)$ is to
find vectors $u,v \in \RR^j$ such that \bes
\begin{cases}
u \ge 0, v\ge 0;\\
v=\theta + R u;\\
u\cdot v =0.
\end{cases}
\ees It is well known (see \cite{DupuisRamanan99} and
\cite{BudhirajaDupuis99}) that with $R=[d_1,...,d_m]$, under Condition \ref{skolip}, for
every $\theta \in \RR^m$, the LCP for $(R,\theta)$ admits a unique
solution $(u,v)\equiv(\cll_m^1(R,\theta),\cll_m^2(R,\theta))$, and furthermore
$\cll_m^2(R,\theta)=\cls(0,\theta)$. Thus the evaluation of
$\cls(0,\theta)$ reduces to solving the above LCP for which numerous
algorithms are available. In the examples considered in the current
work we used a quadratic programming algorithm. Evaluation of
$\cls(x,\theta)$ for $x\neq 0$ can be carried out using a
localization procedure as follows.

Fix $x\in G$ and let $J = \text{In}(x)=\{j \in \{1,...,m\}| \langle
x,e_j\rangle=0 \}$. Let $P_J =\{z \in \RR^m| \langle z,e_j\rangle=0,
\ \forall j\in J^c \}$. Let $\pi_J: \RR^m\rightarrow P_J$ be the
orthogonal projection: $$ \pi_J(z)=z-\sum_{j\in J^c}\langle
z,e_j\rangle e_j.$$ Suppose that $|J|=p$ and $J=\{i_1,...,i_p\}$.
Define a $p\times p$ matrix $R_J$ be the relation $R_J(k,l)=(\pi_J
d_{i_l})_{i_k}$, for $k,l=1,...p$. Let $u_J, v_J\in \RR^p$ be the
solution of LCP for $(R_J,\pi_J \theta)$, i.e., $(u_J,
v_J)=(\cll_p^1(R_J,\pi_J \theta),\cll_p^2(R_J,\pi_J \theta))$.
Once again unique solvability of LCP for $(R_J,\pi_J \theta)$ is assured from Condition \ref{skolip}. Denote $u_J=(\eta_1,...,\eta_p)$ and define $x_1(t)=x+\theta t + t
\sum_{j =1}^p\eta_jd_{i_j}$. Let \[\tau_1=\inf\{t \ge
0|\text{In}(x_1(t))\neq\text{In}(x)\}.\] We define $\tau_1 = \infty$ if the above set is empty. Then $\Gamma(x+\theta i)(t)=x_1(t)$ for all $t<\tau_1$. If $\tau_1<\infty$ set the initial point to be $x_1=x_1(\tau_1)$ and define the trajectory $\{x_2(t)\}_{t\ge 0}$ in a similar way as $\{x_1(t)\}$ by replacing $x$
with $x_1$. Set $\tau_2=\inf\{t \ge 0|\text{In}(x_2(t))\neq \text{In}(x_1)\}$.
Then \[\Gamma(x+\theta i)(\tau_1+t)=\Gamma(x_1+\theta
i)(t)=x_2(t), \mbox{ for all } t<\tau_2. \]  Define now recursively trajectory
$\{x_j(t)\}$ with time points $\tau_j$ and end points $x_j(\tau_j)$,
$j=3,4,...$. Let $j_0$ be such that $\sum_{i=1}^{j_0}\tau_i <1\le
\sum_{i=1}^{j_0+1}\tau_i$. Then
$$\cls(x,\theta)=\Gamma(x+\theta i)(1)=\Gamma(x_{j_0}+\theta
i)(1-\sum_{i=1}^{j_0}\tau_i).$$ Thus the evaluation of
$\cls(x,\theta)$ can be carried out by recursively solving a
sequence of LCP problems.

One difficulty in implementing the above scheme is the possibility
that $\sum_{i=1}^{\infty}\tau_i \le 1$. However using regularity  property of the Skorokhod map, we see that this occurs only
when $\cls(x,\theta)$ is zero. Thus in the practical
implementation of the algorithm we fix a finite threshold $L$ and
carry out the above recursive procedure at most $L$ times and set
$\cls(x,\theta)=0$ if $\sum_{i=1}^{L}\tau_i < 1$.

\subsection{Results.}\label{numerics}
\subsubsection{A 3-d Example with Product Form Stationary Distribution.}
Let $m=3$ and suppose that the reflection matrix is of the form $R=I+Q$, where $I$ is the
identity matrix, and $Q$ is given as
\[ Q=\left[ \begin{array}{ccc}
  0& 0.1&  -0.2\\
	-0.1&    0&      0\\
	  0.2&    0&     0
\end{array} \right].\] 
It can be checked that the spectral radius of $Q$ is less than 1, and so Conditions \ref{skolip} and \ref{completeS} hold.
Take the drift function $b(x)=[-1/2,-1/2,-1/2]^T$ and
$\sigma(x) = I$,  $x \in \RR^3_+$.  The stationary distribution $\nu$ for this example is of product form (see 
\cite{HarWil87}):
 $\exp(1.1667)\otimes \exp(1.0938)\otimes\exp(0.8537)$,
  where $\exp(\mu)$ is the exponential distribution with parameter $\mu$. In implementing the above numerical scheme, we set our initial point
to $X_0=[1,1,1]^T$, and simulate $\{X_k\}_{k=1}^n$ defined by equation
\eqref{scheme2012}, taking $U_k \sim \cln(0,I)$, $\step_k = 1/\sqrt{k}$ and
$n=10^7$. Figure \ref{fig:x1} shows the comparison between the exact
distribution with the first-coordinate marginal of the measure $\nu_n$. 

\begin{figure}[!htb]
\hspace{-0.7in}\begin{center}
 \includegraphics[width=0.6\linewidth]{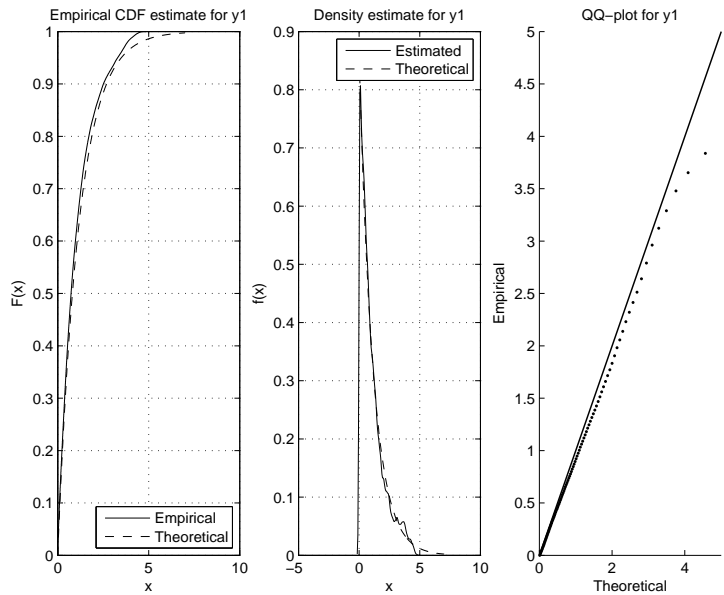}
   \caption{\label{fig:x1}
   Comparison between the nnumerically computed distribution with exact distribution. The left figure shows the comparison between the empirical cumulative distribution function (cdf) and the exact cdf. The middle figure makes a comparison between the estimated density function and the exact density function. And the right figure is the qq-plot of the empirical quantiles versus the exact quantiles.
   }\end{center}
\end{figure}

\subsubsection{Effect of Choice of $\{\lambda_k\}$.}
Consider a two-dimensional SRBM with covariance matrix
$\sigma(x)=I$, drift vector $b(x)=[-1,0]^T$ and reflection matrix
\bes R=\begin{pmatrix} 1 & 0 \\ -1 & 1 \end{pmatrix}.\ees This example was
considered in \cite{DaiHarr92}. We consider the first moment of
the $x_1$-coordinate. The exact value for this moment is known to be 0.5. We
consider $\step_k = k^{-\alpha}$ and examine the influence of the choice of
$\alpha$ on the numerical performance. The results are given in Figure \ref{fig:2dsrbm}. We find that $\alpha = 0.5$ gives the best numerical
convergence.

\begin{figure}[!htb]
\hspace{-0.7in}\begin{center}
 \includegraphics[width=0.6\linewidth]{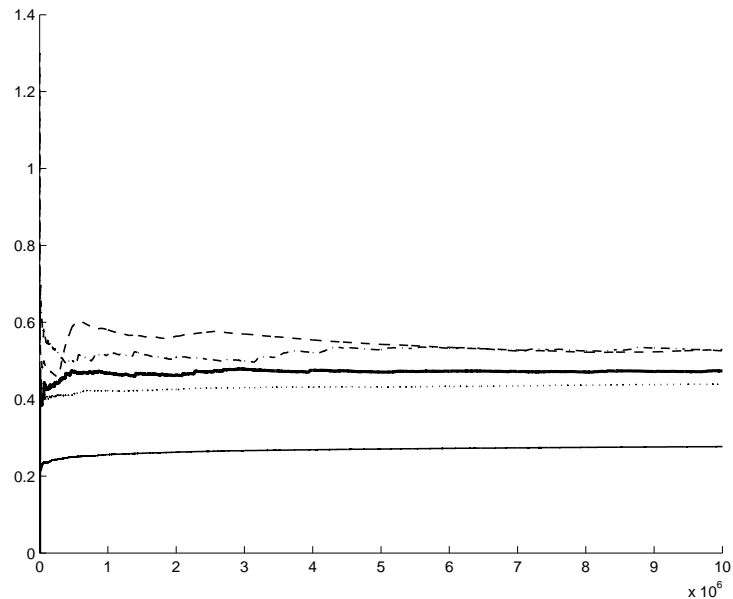}
   \caption{\label{fig:2dsrbm} We consider time step sequence $\step_n =
n^{-\alpha}$ with different choice of $\alpha$ and study the influence of
$\alpha$ on numerical convergence. The thin solid line, the dotted line, the thick solid line, the
   dash-dot line, and the dashed line correspond to $\alpha=$0.1,
   0.3, 0.5, 0.7, and 0.9 respectively. The x-axis shows the value of $n$ while the
y-axis corresponds to $\int x_1\nu_n (dx)$.
   }\end{center}
\end{figure}

\subsubsection{An 8-d symmetric SRBM.}
A SRBM is said to be symmetric if its covariance matrix
$\Gamma$, drift vector $\mu$ and reflection matrix $R$ are symmetric
in the following sense: $\Gamma_{ij}=\Gamma_{ji}=\rho$ for $1 \le i
<j \le d$, $\mu_i=-1$ for $1 \le i \le d$ and $R_{ij}=R_{ji}=-r$ for
$1 \le i <j \le d$, where $r \ge 0$ . The positiveness of $\Gamma$ implies $-1/(d-1) <\rho <1$
and the completely-$\cls$ condition of $R$ implies $r(d-1)<1$. In
this case, It is known (see \cite{DaiHarr92}) that, the first moment of
each of the component is the same, and is given by the following
formula
$$m_1=\frac{1-(d-2)r+(d-1)r\rho}{2(1+r)}.$$

Here we take $d=8$. Then the conditions on the data yield
$-1/7<\rho<1$ and $0 \le r <1/7$. Letting $\rho$ range through
$\{-0.1, -0.05, 0, 0.2, 0.9\}$, and $r$ take value $0.1$, we obtain
estimates of $m_1$ using algorithm in this work. We take $\step_k =
k^{-\alpha}$, $\alpha = 0.5$ and $n=10^7$. The results are shown
in Table \ref{tab:8d}. The results show that as the correlation coefficient
$\rho$ approaches 1, the performance of the algorithm deteriorates.

\begin{table}[!h]
\tabcolsep 0pt \caption{Estimates for $m_1$ when
$d=8$.\label{tab:8d}} \vspace*{-12pt}
\begin{center}
\def\temptablewidth{0.5\textwidth}
{\rule{\temptablewidth}{1pt}}
\begin{tabular*}{\temptablewidth}{@{\extracolsep{\fill}}cccccc}
$\rho$ & -0.1 &-0.05 &0 &0.2 &0.9  \\   \hline
Estimated Val. & 0.131     &  0.137       &     0.163     & 0.414 & 3.205     \\
True Val.     &  0.150      &   0.166     & 0.182 & 0.246 &  0.468  \\
\end{tabular*}
{\rule{\temptablewidth}{1pt}}
\end{center}
\end{table}

\appendix
\section{Appendix}\label{appendix}

\begin{lemma}\label{lemma:supp}
Let $U$ be a random variable with bounded support. Suppose that $\EE U=0$. Then there exists $\alpha \in (0,\infty)$, such that \bes
\EE e^{\lambda U} \le e^{\alpha \lambda^2} \text{ for all } \lambda
\in \RR. \ees
\end{lemma}

\begin{proof}
Without lots of generality we assume that $|U| \le 1$.

Using the convexity of the function $e^{\lambda
x}$, we have \bes e^{\lambda U} \le
\frac{U+1}{2}e^{\lambda}+\frac{1-U}{2}e^{-\lambda}. \ees Taking
expectations in the above inequality and using Taylor's expansion, we have
\bes \EE e^{\lambda U} \le \frac{e^{\lambda}+e^{-\lambda}}{2} \le
e^{\frac{\lambda^2}{2}}. \ees  The lemma then follows on taking
$\alpha = \frac{1}{2}$.
\end{proof}

%

\begin{thebibliography}{}

\bibitem{AtarBudhirajaDupuis01}
R. Atar, A. Budhiraja and P. Dupuis
(2001), On positive recurrence of constrained diffusion process, {\em Annals of Probability} 29(2): 979-1000.

\bibitem{Bas}
G.K. Basak, I. Hu and  C.Z. Wei  (1997), Weak convergence of recursions, {\em Stoch. Proc. App.}, 68: 65--82.


\bibitem{Budhiraja03}
A. Budhiraja (2003), An ergodic control problem for constrained
diffusion processes: Existence of optimal markov control. {\em SIAM
J. Control Optim.} 42, no. 2, 532-558.

\bibitem{BudhirajaBiswas}
A. Budhiraja and A. Biswas, Exit time and invariant measure
asymptotics for small noise constrained diffusions. {\em Stoch. Proc. App.} 121(2011), 899 -- 924.

\bibitem{BudhirajaBorkar04}
A. Budhiraja and V. Borkar (2004), Ergodic control for  constrained
diffusions: Characterization using HJB equations. {\em SIAM J.
Control Optim.} 43, no. 4, 1467-1492.


\bibitem{BudhirajaDupuis99}
A. Budhiraja and P. Dupuis (1999), Simple necessary and sufficient conditions for the stability of constrained processes, {\em SIAM J. Appl. Math.} 59: 1686-1700.


\bibitem{BudhirajaLee07}
A. Budhiraja and C. Lee (2007), Long time asymptotics for constrained diffusions in polyhedral domains, {\em Stochastic Processes and their Applications} 117: 1014-1036.

\bibitem{BudhirajaLee09}
A. Budhiraja and C. Lee (2009), Stationary distribution convergence for generalized Jackson networks in heavy traffic, {\em Math. Oper.  Res.} 34(1): 45-56.


\bibitem{BudhirajaLiu102}
A. Budhiraja and X. Liu (2010), Stability of constrained Markov modulated diffusions,  {\em Submitted}.

\bibitem{CoPaSt09}
R. W. Cottle, J. S. Pang  and R. E. Stone (2009), The linear
complementarity problem, {\em Society for Industrial Mathematics}


\bibitem{DaiHarr92}
J. G. Dai and J. M. Harrison(1992), Reflected Brownian motion in an
orthant: Numerical methods for steady-state analysis, {\em The
Annals of Applied Probability}, 2(1): 65-86.

\bibitem{DaiKur95}
J. G. Dai, Thomas G. Kurtz(1995), A multiclass station with Markovian feedback in heavy traffic, {\em
Mathematics of Operations Research},
Vol. 20, No. 3: 721--742.

\bibitem{DaiKurpre}
J. G. Dai, Thomas G. Kurtz, Characterization of the stationary distribution for a semi-martingale reflecting Brownian motion in a convex polyhedron.
{\em Preprint.}




\bibitem{DaiWilliams95}
J. G. Dai and R. J. Williams (1995), Existence and uniqueness of
semmartingale reflecting Brownian motions in convex polyhedrons,
{\em Theory Probab. Appl.} 40: 1-40.




\bibitem{DupuisIshii91}
P. Dupuis and H. Ishii (1991), On Lipschitz continuity of the solution mapping to the Skorohod problem, with applications, {\em Stochastics} 35: 31-62.

\bibitem{DupuisRamanan99}
P. Dupuis and K. Ramanan (2000), Convex duality and the Skorokhod
Problem. I, {\em  Probability Theory and Related Fields} 115(2):
153-195.



\bibitem{Durrett}
R. Durrett (2010), Probability: theory and examples, {\em Cambridge Univ Press}.

\bibitem{GamZee}
D. Gamarnik and A. Zeevi (2006), Validity of heavy traffic
steady-state approximations in open queueing networks. {\em Ann.
Appl. Probab.} 16(1): 56-90.


\bibitem{HallHeyde80}
P. Hall and C.C. Heyde(1980), Martingale limit theory and its
application, {\em New York: Academic press }.

\bibitem{HarrisonReiman81}
J. M. Harrison and M. I. Reiman (1981), Reflected Brownian motion on an orthant, {\em Annals of Probability} 9(2): 302-308.


\bibitem{HarWil87} J. M. Harrison and R. J. Williams (1987), Brownian models of open queueing networks with
homogeneous customer populations, {\em Stochastics}, 22:77-115.

\bibitem{HarWil872} J. M. Harrison and R. J. Williams (1987), Multidimensional reflected Brownian
motions having exponential stationary distributions, {\em The Annals
of Probability}, 15:115-137.








\bibitem{Kurtz91}
T. G. Kurtz (1991), A control formulation for constrained Markov
processes. {\em Mathematics of Random Media, Lectures in Appl.
Math.} 27: 139-150.

\bibitem{Kushner01}
H. J. Kushner (2001), Heavy traffic analysis of controlled queueing and communication networks. Springer-Verlag Now York, Inc.

\bibitem{LambertonPage02}
D. Lamberton and G. Pag\`es (2002), Recursive computation of the
invariant distribution of a diffusion, {\em Bernoulli}, 8: 367-405.






\bibitem{MandelbaumPats98}
A. Mandelbaum and G. Pats(1998). State-dependent stochastic networks. Part
I: Approximations and applications with continuous diffusion
Limits. {\em Ann. Appl. Probab}. 8(2): 569-646, 1998.

\bibitem{Pell}
M. Pelletier (1998), Weak convergence rates for stochastic
approximation with application to multiple targets and simulated
annealing, {\em Ann. Appl. Probab.}, 8: 10-44.

\bibitem{Peterson91}
W.P. Peterson (1991), A heavy traffic limit theorem for networks of
queues with multiple customer types, {\em Mathematics of Operations
Research}, 16(90-118).


\bibitem{Reiman84}
M. I. Reiman (1984). Open queueing networks in heavy traffic. {\em
Mathematics of Operations Research}, 9(3): 441-458.

\bibitem{ReimanWilliams88}
M. I. Reiman and R. J. Williams (1988), A boundary property of semimartingale reflecting Brownian motions, {\em Probability Theory and Related Fields}, 77: 87-97.


\bibitem{Talay87}
D. Talay (1987), Second order discretization schemes of stochastic
differential systems for the computation of the invariant law, {\em
Stochastics and Stochastics Rep.}, 29: 13-36.

\bibitem{Talay02}
D. Talay (2002), Stochastic Hamiltonian systems: Exponential
convergence to the invariant measure, and discretization by the
implicit Euler scheme, {\em Markov Process. Related Fields}, 8(2):
163-198.

\bibitem{TaylorWilliams93}
L. M. Taylor and R. J. Williams (1993), Existence and uniqueness of semimartingale reflecting Brownian motions in an orthant, {\em Probability Theory and Related Fields} 96: 283-317.

\bibitem{Weiss81}

A. Weiss (1981), Invariant measures of diffusions in bounded
domains, {\em Ph.D. thesis, New York University, New York.}


\bibitem{Williams98_Diff}
R. J. Williams (1998), Diffusion approximations for open multiclass queueing networks: Sufficient conditions involving state space collapse, {\em Queueing Systems: Theory and Applications} 30(1-2): 27-88.


\bibitem{Yamada95}
K. Yamada (1995), Diffusion approximation for open state-dependent
queueing networks in the heavy traffic situation. {\em The Annals of
Applied Probability} 5(4): 958-982.

\end{thebibliography}
\end{document}